\documentclass[11pt,reqno,twoside,a4paper]{article}
\usepackage{fullpage}
\usepackage{mysty}

\title{Stable Phase Retrieval with Mirror Descent}

\author{Jean-Jacques Godeme\thanks{Normandie Univ, ENSICAEN, CNRS, GREYC, France. e-mail: \texttt{jean-jacques.godeme@unicaen.fr, Jalal.Fadili@ensicaen.fr}.} \and Jalal Fadili\footnotemark[1] \and Claude Amra \thanks{Aix-Marseille Univ, CNRS, Centrale Marseille, Institut Fresnel, Marseille, France. \texttt{firstname.lastname@fresnel.fr}.} \and Myriam Zerrad\footnotemark[2]}  
\date{}
\begin{document}    
\maketitle
\begin{abstract}
In this paper, we aim to reconstruct an $n$-dimensional real vector from $m$ phaseless measurements corrupted by an  additive noise. We extend the noiseless framework developed in \cite{godeme2023provable}, based on mirror descent (or Bregman gradient descent), to deal with noisy measurements and prove that the procedure is stable to (small enough) additive noise. In the deterministic case, we show that mirror descent converges to a critical point of the phase retrieval problem, and if the algorithm is well initialized and the noise is small enough, the critical point is near the true vector up to a global sign change. When the measurements are \iid Gaussian and the signal-to-noise ratio is large enough, we provide global convergence guarantees that ensure that with high probability, mirror descent converges to a global minimizer near the true vector (up to a global sign change), as soon as the number of measurements $m$ is large enough. The sample complexity bound can be improved if a spectral method is used to provide a good initial guess. We complement our theoretical study with several numerical results showing that mirror descent is both a computationally and statistically efficient scheme to solve the phase retrieval problem.
\end{abstract}
\section{Introduction}\label{sec:intro}

\subsection{Problem statement and motivations}
The phase retrieval problem arises in many applications including X-ray crystallography, diffraction imaging, and light scattering, to name just a few; see \cite{shechtman_phase_2014,JaganathanReview16,luke_phase_2017} and references therein. Phase retrieval is a very active research area and we refer to \cite{shechtman_phase_2014,JaganathanReview16,luke_phase_2017,fannjiang_numerics_2020,vaswani_non-convex_2020} for recent reviews of the current state of the art.

Our focus in this paper is phase retrieval with possibly noisy measurements. In real applications, the intensity measurements are not perfectly acquired. For instance, let us consider light scattering for precision in optics~\cite{amra_instantaneous_2018} which is our motivating application, where the goal is to describe the roughness of a polished surface. The latter is illuminated with a laser source, and the diffusion is measured by moving a detector. Then the power spectral density of the surface topography can be directly measured. However, during the acquisition process, different  types of noise can corrupt the measurements such as photon noise, thermal noise, Johnson noise, \etc. Knowing the statistical model underlying the noise and the way it contaminates the measurements can prove useful to achieve robust reconstruction. The noise model can then be incorporated as the negative log-likelihood in the minimization objective. There are several noise models used in phase retrieval. One of them is the signal-dependent Poisson noise model which models the photon count noise. Another noise model is the (complex-valued) noise arising from multiple scattering, which can be modelled by the (complex) circularly-symmetric Gaussian distribution, and used to describe Rayleigh fading channels encountered in communication systems. Yet another source of noise is the thermal one or the incoherent background noise. 

In this manuscript, and similarly to \cite{candes_phaselift_2013,demanet_stable_2013,chen_solving_2017}, we will work with a generic additive noise model, without any particular statistical assumption, in which the noisy phase retrieval problem reads
\begin{equation}\tag{NoisyPR}\label{eq:NoisyPR}
\begin{cases}
\text{Recover $\avx\in \bbR^n$ from the measurements $y \in \bbR^m$} \\
y[r]=|\adj{a_r}\avx|^{2}+\epsi[r], \quad r \in \bbrac{m}, 
\end{cases}
\end{equation}
where $[r]$ is the $r$-th entry of the corresponding vector, and $\epsi\in\bbR^m$ is the noise vector. Throughout the paper, $A$ is the $m \times n$ matrix with $\adj{a_r}$'s as its rows.

Since $\avx$ is real-valued, the best one can hope for is to ensure that $\avx$ is uniquely determined from its intensities up to a global sign. Phase retrieval is in fact an ill-posed inverse problem in general, and   even  for $\epsilon=0$, checking whether a solution to \eqref{eq:NoisyPR} exists or not is known to be NP complete \cite{Sahinoglou91}. The situation is even more complicated in presence of noise. Thus, one of the major challenges is to design efficient recovery algorithms and find conditions on $m$, $(a_r)_{r \in {\bbrac{m}}}$ and $\epsilon$ which guarantee stable recovery in presence of noise. This is the goal we pursue in this paper.

In this paper, we cast the noise-aware phase retrieval problem \eqref{eq:NoisyPR} as the smooth but nonconvex minimization problem
\begin{equation}\label{eq:formulepro}
\min_{x\in\bbR^n}\left\{ f(x)\eqdef\qsom{\Ppa{y[r]-|(Ax)[r]|^2}^2}\right\}.
\end{equation}
In fact, this is the same problem as in \eqref{eq:eqminPR} studied in the noiseless case in \cite{godeme2023provable}. There, we proposed a mirror descent algorithm based on a suitably chosen entropy. In particular, we analyzed the case where the measurements were either \iid standard Gaussian measurements or drawn from the Coded Diffraction Pattern (CDP) model. It is our aim in this paper to extend these results to the noisy case and prove stability guarantees for mirror descent to minimize \eqref{eq:eqminPR}.

\subsection{Prior work}\label{sec:prior-work}
In phase retrieval, understanding the impact of noise is crucial because real-world measurements are invariably corrupted by it. Thus, establishing stability of phase retrieval to (small enough) noise is of paramount importance. Stability of phase retrieval to (small enough) noise has been studied by several authors with various measurement ensembles and reconstruction procedures. For convex relaxations, Cand\`es and Li showed in \cite{candes_solving_2014} that a noise-aware variant of PhaseLift is stable against additive noise with a reconstruction error bound $O\Ppa{\normm{\avx},\frac{\norm{\epsi}{1}}{m\normm{\avx}}}$ as soon as $m\gtrsim n$ (complex) Gaussian measurements are taken (see also \cite{candes_phaselift_2013} where the sample complexity was $m\gtrsim n\log(n)$). This is of course only meaningful if the signal-to-noise ration is sufficiently high. This result has been extended to the case of sub-Gaussian measurements in \cite{krahmer_complex_2020}. For nonconvex formulations, Huang and Xu in \cite{huang_performance_2021} analyzed the performance of the Wirtinger flow and showed that any solution of this algorithm enjoys a reconstruction error upper bound $O\Ppa{\min\Bba{\frac{\sqrt{\normm{\epsilon}}}{m^{1/4}},\frac{\normm{\epsi}}{\normm{\avx}\sqrt{m}}}}$ as soon as $m\gtrsim n$.  The amplitude, the reshaped and the  Wirtinger flow algorithms are stable against additive noise as shown respectively in
\cite{zhang_nonconvex_2017},\cite{wang_solving_2017} and \cite{gao_stable_2016}. Indeed, these authors showed that the reconstruction error scales as $ O\Ppa{\frac{\normm{\epsi}}{\sqrt{m}}}$. The convergence result is obtained under the specific assumption that $\norm{\epsilon}{\infty}\lesssim \normm{\avx}$. In \cite{xia_performance_2024}, the authors study the performance  of the amplitude-based model and showed the solution satisfies the following reconstruction upper bound $O\Ppa{\frac{\normm{\epsi}}{\sqrt{m}}}$ as soon as $m\gtrsim n$. The  truncated Wirtinger flow \cite{chen_solving_2017}, which can account even for Poisson noise, has been shown to be stable with a reconstruction error bound that scales as $O\Ppa{ \frac{\normm{\epsi}}{\sqrt{m}\normm{\avx}}}$ under the  assumption $\norm{\epsilon}{\infty}\lesssim\normm{\avx}^2$ provided that $m\gtrsim n$. It was shown there that this is the best statistical guarantee any algorithm can achieve by deriving a fundamental lower bound on the minimax estimation error.

\subsection{Contributions and relation to prior work}\label{sec:contribution}
In this paper, we claim that mirror descent to solve \eqref{eq:eqminPR} is stable against sufficiently small additive noise. This in turn provides recovery error bounds of the noisy phase retrieval problem \eqref{eq:NoisyPR}. In the deterministic case, we show that for almost all initializers, mirror descent converges to a critical point near the true vectors (up to sign ambiguity) where the objective has no direction of negative curvature. In the random case, we consider \iid Gaussian measurements, and in the regime where the signal-to-noise ratio is large enough (see Assumption~\ref{ass:SNR}), we provide a complete geometric characterization of the landscape of the nonconvex objective provided that $m \gtrsim n\log^3(n)$. In turn, this allows us to describe the set of the critical points of $f$ as the union of the strict saddle points and global minimizers of $f$. From this, we provide a global convergence to a point in $\Argmin(f)$, which is near $\avx$ (up to sign ambiguity), as soon as the number of samples is large enough. If $m \gtrsim n \log(n)$, using a spectral initialization method, we provide a local convergence to a vector in the neighborhood of the target vector (up to sign ambiguity). By "near" we mean a reconstruction error that eventually scales as $O\Ppa{ \frac{\normm{\epsi}}{\sqrt{m}\normm{\avx}}}$ which matches the minimax optimal bounded established in \cite[Theorem~3]{chen_solving_2017}. Compared to the Wirtinger flow and variants, our algorithm, by adapting to the geometry, offers an easier and dimension-independent choice of the parameters (in fact one, the descent-step size), and has global convergence guarantees.

Our results can be easily extended to sub-Gaussian measurements with minor changes. The case where $a_r$'s are drawn from the CDP model is, however, far more challenging. Indeed, this model enjoys less randomness compared to the (sub-)Gaussian case and many of our arguments that require the uniformization of some bounds that are difficult to extend to the CDP model. Nevertheless, numerical experiments suggest that stable recovery still holds for our mirror descent algorithm with CDP measurements.

\subsection{Outline of the paper}\label{sec:outline}

The rest of the paper is organized as follows. Section~\ref{sec:mdbtny_prelim} gathers preliminary material necessary to our exposition.  In Section~\ref{sec:mdbtny_deter}, we recall the mirror descent algorithm with backtracking and establish its global and local convergence guarantees in the deterministic case. In Section~\ref{sec:mdbtny_rand}, we sample complexity bounds with Gaussian measurements for our deterministic guarantees to hold with high probability. Section~\ref{sec:mdnsy_numexp} describes the numerical experiments. The proofs of technical results are deferred to the appendix. In particular, Section~\ref{landscapestudy} studies the landscape of the noise-aware objective with Gaussian measurements.

\section{Preliminaries}\label{sec:mdbtny_prelim}
\subsection{Notations}
We denote $\pscal{\cdot,\cdot}$ the scalar product and $\normm{\cdot}$ the corresponding norm. $B(x,r)$ is the corresponding ball of radius $r$ centered at $x$ and $\mathbb{S}^{n-1}$ is the corresponding unit sphere. For $m \in \N^*$, we use the shorthand notation ${\bbrac{m}}=\{1,\ldots,m\}$.  The $i$-th entry of a vector $x$ is denoted $x[i]$. Given a matrix $M$, $\transp{M}$ is its transpose and $\adj{M}$ is its adjoint (transpose conjugate) when it is complex. Let $\lambda_{\min}(M)$ and $\lambda_{\max}(M)$ be respectively the smallest and the largest eigenvalues of $M$. For two real symmetric matrices $M$ and $N$, $M \slon N$ if $M-N$ is positive semidefinite.

$\dom(f)$ is the domain of the function $f$. {$f^*$ denotes the Legendre-Fenchel conjugate of $f$.} 
Recall that the set of critical points of a differentiable function $f$ is $\crit (f) = \enscond{x \in \bbR^n}{\nabla f(x)=0}$.

Let us denote the set of true vectors by $\xoverline{\calX}=\{\pm\avx\}$. For any vector $x\in\bbR^n$, the distance to the set of true vectors is  
\begin{equation}
\dist(x,\xoverline{\calX})\eqdef \min\pa{\normm{x-\avx},\normm{x+\avx}} .
\end{equation}
Our limitation of the ambiguity set to $\{\pm\avx\}$ may appear restrictive since even for real vectors, the equivalence class is much larger than what we are allowing. However, our restriction will be justified in the oversampling regime.

\subsection{Bregman toolbox}
\paragraph{Definition and properties}
Since our focus is on phase retrieval, we will restrict our discussion here to entropy/kernel functions $\phi$ that are strictly convex and differentiable on the whole space. As such, $\phi$ is in fact a Legendre function; see \cite[Chapter~26]{rockafellar_convex_1970}.
Observe that following \cite[Theorem~26.5]{rockafellar_convex_1970}, a function is Legendre if and only if its Fenchel conjugate $\phi^*$ is Legendre. Moreover $\nabla\phi$ is a bijection from $\inte(\dom(\phi))=\bbR^n$ to $\inte(\dom(\phi^*))$ with $\nabla\phi^*=(\nabla\phi)^{-1}$.  

For any function differentiable function $\phi:\bbR^n \to \bbR$, its Bregman divergence is
\begin{equation}
D_{\phi}(x,z)=\phi(x)-\phi(z)-\pscal{\nabla \phi(z),x-z} . 
\end{equation}
This proximity measure is not a distance (it is not symmetric in general for instance nor it verifies the triangle inequality). If $\phi(x)=\frac{1}{2}\normm{x}^2$, the Bregman divergence is the usual euclidean distance $D_{\phi}(x,z)=\frac{1}{2}\normm{x-z}^2$. 
	
Throughout the rest of the work, we use the following properties of the Bregman divergence.
\begin{proposition}\label{pro:bregman} \textbf{(Properties of the Bregman distance)}
	\begin{enumerate} [label=(\roman*)]
		\item \label{pro:bregman1} The Bregman divergence of $\phi$ is nonnegative if and only if $\phi$ is convex. If in addition, $\phi$ is strictly convex, then its Bregman divergence vanishes if and only if its arguments are equal. \label{pp:bregman1}
		
		\item \label{pro:bregman2}
		Linear additivity: for any $\alpha,\beta \in \bbR$ and any differentiable functions $\phi_1$ and $\phi_2$ we have 
		\begin{equation}\label{linear}
		D_{\alpha \phi_1+\beta \phi_2}(x,z)=\alpha D_{\phi_1}(x,z) + \beta D_{\phi_2}(x,z),
		\end{equation}
		for all $x,z \in \bbR^n$.
		\item \label{pro:bregman3} The three-point identity:  For any $x,u,z \in \bbR^n$, we have 
		\begin{equation}\label{3poin}
		D_{\phi}(x,z)-D_{\phi}(x,u)-D_{\phi}(u,z)=\pscal{\nabla\phi(u)-\nabla\phi(z),x-u}.
		\end{equation}
		
		\item \label{pro:bregman4}
		Suppose that $\phi$ is also $C^2$ and $\nabla^2 \phi(x)$ is positive definite for any $x$. Then for every convex compact subset $\Omega\subset\bbR^n$, there exists $0<\theta_{\Omega}\leq\Theta_{\Omega}<+\infty$ such that for all $x,z \in \Omega$,
		\begin{equation}\label{eq:striconv}
		\frac{\theta_{\Omega}}{2}\normm{x-z}^2\leq D_{\phi}(x,z)\leq\frac{\Theta_{\Omega}}{2}\normm{x-z}^2.
		\end{equation}
	\end{enumerate} 
\end{proposition}

\paragraph{Regularity of functions}\label{sec:regul-func}
The following definition extends the classical gradient Lipschitz continuity property to the Bregman setting. This notion is coined "relative smoothness" and is important for the analysis of optimization problems with objective functions that are differentiable but lack the popular gradient Lipschitz-smoothness. The earliest reference to this notion can be found in an economics paper \cite{birnbaum2011distributed} where it is used to address a problem in game theory involving fisher markets. Later on it was developed in \cite{bauschke_descent_2016,bolte_first_2017} and then in \cite{lu2018relatively}, although first coined relative smoothness in \cite{lu2018relatively}. 
\begin{definition}\label{smoothadaptable}\textbf{($L-$relative smoothness)} Let $\phi$ and $g$ be differentiable functions. $g$ is called $L-$smooth relative to $\phi$ if there exists $L>0$ such that $L\phi-g$ is convex on $\bbR^n$, \ie for all $x,z \in \bbR^n$,
	\begin{equation}\label{eq:smoothadaptable}
	D_g(x,z) \leq LD_{\phi}(x,z) .
	\end{equation}
\end{definition}

When $\phi$ is the energy entropy, \ie $\phi=\frac{1}{2}\normm{\cdot}^2$, \eqref{eq:smoothadaptable} is nothing but a manifestation of the standard descent lemma implied by Lipschitz continuity of the gradient of $g$.

In a similar way, we also extend the standard local strong convexity property to a relative version \wrt to an entropy or kernel $\phi$.

\begin{definition}\label{relativeconvex}\textbf{(Local relative strong convexity)}
	Let $\C$ be a non-empty subset of $\bbR^n$ and $\phi$ and $g$ be differentiable functions. For $\sigma > 0$, we say that $g$ is $\sigma$-strongly convex on $\C$ relative to $\phi$ if 
	\begin{equation}
	D_g(x,z)\geq \sigma D_{\phi}(x,z) \quad\forall\quad (x,z) \in \C^2 .
	\end{equation}
\end{definition}
When $\C=\bbR^n$, we recover the notion of global relative strong convexity. If $\phi$ is the energy entropy (\ie $\phi=\frac{1}{2}\normm{\cdot}^2$), one recovers the standard definition of (local/global) strong convexity.

The idea of global relative strong convexity has already been used in the recent literature, see \eg  \cite[Proposition~4.1]{Teboulle18} and \cite[Definition~3.3]{Bauschke19}. Its local version was first proposed in \cite{Silveti22}. Relation of global relative strong convexity to gradient dominated inequalities, which is an essential ingredient to prove global linear convergence of mirror descent, was studied in \cite[Lemma~3.3]{Bauschke19}.

We record the following simple lemma which will be useful to compare Bregman divergences of smooth functions through the partial Loewner ordering of their respective hessians.
\begin{lemma}\label{lem:bregcomp}
	Let $g,\phi\in C^2(\bbR^n)$. If $\forall u\in\bbR^n$, $\nabla^2g(u) \lon \nabla^2\phi(u)$ for all $u$ in the segment $[x,z]$, then,
	\begin{equation}
	D_g(x,z)\leq D_{\phi}(x,z) .
	\end{equation}
\end{lemma}
\begin{proof}
	The result comes from the Taylor-MacLaurin expansion. Indeed we have $\forall x,z\in \bbR^n$ 
	\begin{align*}
		D_g(x,z)
		&=g(x)-g(z)-\pscal{\nabla g(z),x-z} \\
		&=\int_0^1(1-\tau)\pscal{x-z,\nabla^2g(z+\tau(x-z))(x-z)}d\tau,
	\end{align*}
	and thus
	\begin{multline*}
		D_{\phi}(x,z)-D_{g}(x,z) = \\ \int_0^1(1-\tau)\pscal{x-z,\Ppa{\nabla^2\phi(z+\tau(x-z))-\nabla^2g(z+\tau(x-z))}(x-z)}d\tau .
	\end{multline*}
	The positive semidefiniteness assumption implies the claim.
\end{proof}

\section{Deterministic Stable Recovery}\label{sec:mdbtny_deter}
\subsection{Mirror descent with backtracking}
Observe that the objective in \eqref{eq:NoisyPR} can be decomposed as
\begin{equation}\label{eq:eqminPR}
f(x)=\qsom{\loss{|(Ax)[r]|^2}{|(A\avx)[r]|^2-\epsi[r]}} .
\end{equation}
The objective $f$ is $C^2(\bbR^n)$ and nonconvex (in fact only weakly convex). Moreover, its gradient is not Lipschitz continuous. However, using the strongly convex entropy (see \cite[Proposition~2.6]{godeme2023provable} for its properties),
\begin{align}\label{eq:KGDpr}
	\psi(x)=\frac{1}{4}\normm{x}^4+\frac{1}{2}\normm{x}^2 ,
\end{align}
$f$ turns out to be smooth relative to $\psi$.
\begin{lemma}\label{Tsmad}
	Let $f$ and $\psi$ defined in \eqref{eq:formulepro}-\eqref{eq:KGDpr}. $f$ is $L-$smooth relative to $\psi$ on $\bbR^n$ for any $L\geq \frac{1}{m}\sum\limits_{r=1}^m\normm{a_r}^2\para{3\normm{a_r}^2+\norm{\epsi}{\infty}}$. 
\end{lemma}
See Section~\ref{PrTsmad} for the proof. This estimate of the modulus of relative smoothness $L$ is crude and depends also on the noise. This estimate will be largely improved for Gaussian measurements.\\

This relative smoothness property is the key motivation behind considering the framework of mirror descent or Bregman gradient descent. The mirror descent scheme with backtracking is detailed in Algorithm~\ref{alg:MDBT}.

\begin{algorithm}[H]
 \caption{Mirror Descent for Phase Retrieval}
 \label{alg:MDBT}
    \textbf{Parameters:} $0 < L_0 \leq L$ (see Lemma~\ref{Tsmad}), $\kappa \in ]0,1[$, $\xi\leq1$ . \\
    \textbf{Initialization:} $x_0\in\bbR^n$\;
    \For{$k=0,1,\ldots$}{
    \Repeat{$D_{f}(\xkp,\xk) \leq \xi L_{k}D_{\psi}(\xkp,\xk)$}{
    $\gak=\frac{1-\kappa}{\lk}$; \\
    $\xkp=F(\xk)=\nabla\psi^{*}\paren{\nabla\psi(\xk)-\gak\nabla f(\xk)}$; \\
    $\lk \leftarrow \lk/\xi$; 
    }
    $L_{k+1} \leftarrow \xi\lk$; \\
    \textbf{Output:} $\xkp$.
    }
\end{algorithm}     
The pair $(f,\psi)$ defined in \eqref{eq:formulepro}-\eqref{eq:KGDpr} satisfies \cite[Assumptions $A,B,C,D$]{bolte_first_2017} and thus the mapping $F$ in Algorithm~\ref{alg:MDBT} is well-defined and single-valued on $\bbR^n$. Moreover, the mirror step $\nabla \psi^*(z)$ can be computed easily as $\nabla \psi^*(z) = t^* z$, where $t^*$ is the unique real positive root of the third-order polynomial $t^3\normm{z}^2+t-1=0$; see \cite[Proposition~2.8]{godeme2023provable}.

\subsection{Deterministic recovery guarantees by mirror descent}
Before the deterministic result, we start by recalling the notion of strict saddles.  

\begin{definition}[Strict saddle point]
	A point $\xpa\in\crit(f)$ is a strict saddle point of $f$ if $\lambda_{\min}(\nabla^2f(\xpa))< 0$. The set of strict saddle points of $f$ is denoted $\strisad(f)$.
\end{definition}


We now claim that mirror descent is stable against additive noise, as demonstrated in
the following theorem.
\begin{theorem}\label{thm:Det_reslt}
	Consider the noisy phase retrieval problem cast as \eqref{eq:formulepro}. Let $\seq{\xk}$ be a bounded sequence generated by Algorithm~\ref{alg:MDBT}. Then,
	\begin{enumerate}[label=(\roman*)]
		\item \label{point-i} the sequence $\seq{\xk}$ has a finite length, converges to a point in $\crit(f)$ and the values $\seq{f(\xk)}$ are nonincreasing.
		
		Take $L_k = L, \forall k \geq 0$. Then,
		
		\item \label{point-ii} for Lebesgue almost all initializers $\xo$, the sequence $\seq{\xk}$ converges to a critical point which cannot be a strict saddle, \ie $x_k \to \widetilde{x} \in \crit(f)\bsl\strisad(f)$.
		
		\item \label{point-iii} Assume that $\Argmin(f) \neq \emptyset$. Let $\rho, \sigma > 0$ such that $\rho>\frac{\sqrt{2}\normm{\epsi}}{\sqrt{m\sigma}}$ and define the radius $r \leq \sqrt{\frac{\rho^2-\frac{2\normm{\epsi}^2}{m\sigma}}{\max\Ppa{\Theta(\rho),1}}}$. If the initial point $\xo \in B(\overline{\calX},r)$ and $f$ is $\sigma$-strongly convex relative to $\psi$ on $B(\overline{\calX},\rho)$ then $\xk\in B(\overline{\calX},\rho), \forall k \in\bbN$, and
		\begin{equation}\label{eq:loclinearnoisy}
		\dist^2\pa{\xk,\overline{\calX}} \leq \Ppa{1-\gamma\sigma}^{k-1}\rho^2+2\frac{\normm{\epsi}^2}{m\sigma},
		\end{equation}
	\end{enumerate}
\end{theorem}
See Section~\ref{PrMaintheo} for the proof.\\

Some remarks are in order.
\begin{remark}\label{rem:deterministic} {\ }
	\begin{itemize}
		\item Clearly, claim~\ref{point-i} suggests that even in the presence of the noise, any bounded sequence of Algorithm~\ref{alg:MDBT} will converge to a critical point of $f$ with decreasing values. Let us observe that the sequence generated by our algorithm is bounded if for instance $f$ is coercive, in which case $\Argmin(f)$ is also a non-empty compact set. This happens to be true when $A$ is injective, \ie in the oversampling regime as we will show in the random case. 
		\item Claim~\ref{point-ii} shows that when the initial guess $\xo$ is chosen according to a distribution that has a density \wrt the Lebesgue measure with constant step-size, then the sequence generated by mirror descent converges to a critical point where $f$ has no direction of negative curvature.  
		\item Concerning our local results in claim-\ref{point-iii}, if mirror descent is well initialized \ie, in a ball of sufficiently small radius $r < \rho$ around the true vectors $\overline{\calX}$, and if $f$ is strongly convex relative to $\psi$ on the larger ball $B(\overline{\calX},\rho)$, then all the iterates $\seq{\xk}$ will remain in $B(\overline{\calX},\rho)$. Moreover, the sequence $(\xk)_{k \in \N}$ will converge to a critical point $\wtilde{x}$ obeying
		\[
		\dist(\wtilde{x},\overline{\calX}) \leq \frac{\sqrt{2}\normm{\epsi}}{\sqrt{m\sigma}} .
		\]
	\end{itemize}
\end{remark}

\section{Stable Recovery from Gaussian Measurements}\label{sec:mdbtny_rand}
The deterministic stable recovery results of Theorem~\ref{thm:Det_reslt}\ref{point-ii}-\ref{point-iii} require for instance a local relative strong convexity condition around $\pm\avx$ and possibly a good enough initial guess. A natural question to ask is when these conditions hold true. In turns out that this is indeed the case in the oversampling regime with \iid Gaussian measurements, and if the noise is small enough. This section is devoted to rigorously show these statements. 

We consider that the sensing vectors $(a_r)_{r \in \bbrac{m}}$ are drawn \iid from a real zero-mean standard Gaussian distribution. We also work under the following assumption on the noise $\epsi$. 
\begin{assumption}{\ } \label{ass:SNR}
	Denote $\widetilde{\epsi}=\frac{1}{m}\sum\limits_{r=1}^{m}\epsi[r]$. Given $\lambda \in ]0,1[$, we suppose that
	\begin{equation}\label{eq:cond_r}
	\begin{gathered}
	0 \leq \frac{\widetilde{\epsi}}{\min\pa{\normm{{\avx}}^2,1}} < \lambda \qandq 
	\norm{\epsilon}{\infty} \leq c_s \min\pa{\normm{{\avx}}^2,1}, \\ 
	\text{ for some constant } c_s \in \left]0,\frac{(1-\lambda)\normm{\avx}\sqrt{\lambda\min\pa{\normm{{\avx}}^2,1}-\widetilde{\epsi}}}{2\sqrt{6}\min\pa{\normm{{\avx}}^2,1}} \right [. 
	\end{gathered}
	\end{equation} 
\end{assumption}

To get better understanding of this assumption, we observe that it implies that
\begin{eqnarray*}
	\frac{\normm{\epsi}}{\sqrt{m}\normm{{\avx}}^2} \leq \frac{\norm{\epsi}{\infty}}{\normm{{\avx}}^2} < \frac{(1-\lambda)\sqrt{\lambda}}{2\sqrt{6}} < 1 .
\end{eqnarray*}
On the other hand, for the observation model \eqref{eq:NoisyPR} with \iid real Gaussian sensing vectors, the signal-to-noise ratio (SNR) is captured by
\[
\SNR \eqdef \frac{\sum_{r=1}^m|\transp{a_r}\avx|^4}{\normm{\epsi}^2} \approx \frac{3m\normm{\avx}^4}{\normm{\epsi}^2} .
\]
In other words, Assumption~\ref{ass:SNR} amounts to imposing that the SNR is large enough, \ie
\[
\sqrt{\SNR} \gtrsim \frac{6\sqrt{2}}{(1-\lambda)\sqrt{\lambda}}.
\]


	Let us also observe that Assumption~\ref{ass:SNR} imposes that the empirical mean $\widetilde{\epsi}$ is non-negative. This is a practical assumption that is fulfilled in many applications and will be helpful to describe the landscape of the noise-aware objective for Gaussian measurements. However, it was not used to have the deterministic guarantees.  

\medskip 

Some of our stable recovery guarantees will be local provided that Algorithm~\ref{alg:MDBT} is initialized with a good guess. For this, we will use a spectral initialization method; see for instance \cite{Candes_WF_2015,chen_solving_2017,netrapalli_phase_2015,zhang_nonconvex_2017,wang_solving_2017,waldspurger_phase_2018}. The procedure consists of taking $\xo$ as the leading  eigenvector of a specific matrix as described in Algorithm~\ref{alg:algoSPIniit}.

\begin{algorithm}[H]\caption{Spectral Initialization.}
	\label{alg:algoSPIniit}
	\KwIn{$y[r], r=1,\ldots,m$}
	\KwOut{$\xo$}
	Set $\lambda^2=n\frac{\sum\limits_ry[r]}{\sum\limits_r\normm{a_r}^2}=n\para{\frac{\sum\limits_r\pscal{a_r,\avx}^2}{\sum\limits_r\normm{a_r}^2}+\frac{\sum\limits_r\epsi_r}{\sum\limits_r\normm{a_r}^2}}$ \;
	Take $\xo$ the top eigenvector of $Y=\frac{1}{m}\sum\limits_{r=1}^my[r]a_r\transp{a_r}$ normalized to $\normm{\xo}=\lambda$.
\end{algorithm}


To lighten notations and clarify our proof, we consider the following events on whose intersection our deterministic convergence result will hold with high probability. Let fix $\vrho\in]0,1[$ and $\lambda\in\left]\frac{1}{9\sqrt{2}},1\right[$.
\begin{itemize}
	\item The event 
	\begin{equation}\label{eq:crit_char_G}
	\calE_{\rm strictsad} = \ensB{\crit(f)=\Argmin(f)\cup\strisad(f)} 
	\end{equation}
	means that the set of critical points of the function $f$ reduces to the set global minimizers of $f$ and the set of strict saddle points. 
	
	\item The event 
	\begin{equation} \label{eq:uniconcen_G}
	\calE_{\rm conH} = \left\{\forall x\in\bbR^n,\quad\normm{\nabla^2f(x)-\esp{\nabla^2f(x)}}\leq\vrho\para{\normm{x}^2+\normm{\avx}^2/3+\norm{\epsi}{\infty}}\right\}  
	\end{equation}
	captures the deviation of the Hessian of $f$ around its expectation.
	
	\item The event
	\begin{equation}\label{eq:injectivity_G}
	\calE_{\rm inj} =\left\{\forall x\in\bbR^n,\quad \Ppa{1-\vrho}\normm{x}^2\leq \frac{1}{m}\normm{A x}^2\right\}
	\end{equation}
	corresponds to the injectivity of the measurement operator $A$.

	\item Let us denote by $\calE_{\rm smad}$ the event on which the function $f$ is $L-$smooth relative to $\psi$ with $L=3+\widetilde{\epsi}+\vrho\max\pa{\normm{\avx}^2/3+1,1}$. 
	
	\item Let us define $\rho=\frac{(1-\lambda)\normm{\avx}}{\sqrt{3}}>0$ and $\calE_{\rm scvx}$ is the event on which $f$ is $\sigma$-strongly convex relative to $\psi$ locally on $B(\xoverline{\calX},\rho)$, with $\sigma=\lambda\min\pa{\normm{\avx}^2,1}-\widetilde{\epsi}-\vrho\max\pa{\normm{\avx}^2/3+\norm{\epsi}{\infty},1}$. 
	
	\item We end up denoting
	\begin{equation}\label{eq:e1_conv_G}
	\calE_{\rm conv}=\calE_{\rm strictsad}\cap\calE_{\rm conH}\cap\calE_{\rm inj}\cap\calE_{\rm smad}\cap\calE_{\rm scvx}. 
	\end{equation}
\end{itemize}

Our main result for Gaussian measurements is the following. 
\begin{theorem}\label{thm: Gauss}
	Fix $\lambda\in\left]\frac{1}{9\sqrt{2}},1\right[$ and $\vrho\in \left]0,\frac{\lambda\min(\normm{\avx}^2,1)-\widetilde{\epsi}}{2\max\pa{\normm{\avx}^2/3+\norm{\epsi}{\infty},1}}\right[$. Let $\seq{\xk}$ be the sequence generated by Algorithm~\ref{alg:MDBT}. Under Assumption~\ref{ass:SNR}, let us  define  for any $\kappa\in]0,1[$ 
	\[
	\nu=\frac{(1-\kappa)\Ppa{\lambda\min\pa{\normm{\avx}^2,1}-\widetilde{\epsi}-\vrho\max\pa{\normm{\avx}^2/3+\norm{\epsi}{\infty},1}}}{3+\widetilde{\epsi}+\vrho\max\pa{\normm{\avx}^2/3+\norm{\epsi}{\infty},1}} \in [0,1[
	\]
	and
	\[
	\varsigma=\frac{2\sqrt{2}\normm{\epsi}}{\sqrt{m\Ppa{c_s\min\pa{\normm{\avx}^2,1}-\widetilde{\epsi}}}}.
	\]
	\begin{enumerate}[label=(\roman*)]
		\item \label{thm: Gauss-i} If the number of measurements $m$ is large enough \ie, $m\geq C(\vrho)n\log^3(n)$, then for almost all initializers $\xo$ of Algorithm~\ref{alg:MDBT} with the step-size $\gamma\equiv\frac{1-\kappa}{3+\widetilde{\epsi}+\vrho\max\pa{\normm{\avx}^2/3+\norm{\epsi}{\infty},1}}$, we have 
		\[
		\xk \to \xsol\in\Argmin(f)\cap B\Ppa{\overline{\calX},\varsigma}
		\]
		and $\exists K>0$ such that  for all $k\geq K$, we have 
		\begin{equation}\label{noisy_conv_lin_G}
		\dist^2(\xk,\overline{\calX})\leq \frac{\normm{\avx}^2}{3}(1-\nu)^{k-K}+\varsigma^2 . 
		\end{equation}
		Both above events hold with a probability at least $1-2e^{-\frac{m\pa{\sqrt{1+\vrho}-1}^2}{8}}-e^{-\Omega(m)}-5e^{-\zeta n}-4/n^2 -c/m$, where $c, \zeta$ are fixed numerical constants.
		\item \label{thm: Gauss-ii} Suppose that $\vrho$ obeys \eqref{eq:spectralinitvrhobnd_nsy}. If $m$ is such that $m\geq C(\vrho,\norm{\epsi}{\infty})n\log(n)$, and Algorithm~\ref{alg:MDBT} is initialized with the spectral method in Algorithm~\ref{alg:algoSPIniit}, then \eqref{noisy_conv_lin_G} holds for all $k \geq K=0$ with probability at least $1-2e^{-\frac{m\pa{\sqrt{1+\vrho}-1}^2}{8}}-e^{-\Omega(m)}-5e^{-\zeta n}-4/n^2$, where $\zeta$ is a fixed numerical constant.
	\end{enumerate}
\end{theorem}

The choice of parameters can be made easier to grasp when $\normm{\avx}=1$ as assumed in many works.

\begin{corollary}\label{cor:Gauss}
	Suppose that $\normm{\avx}=1$ and the noise is small enough. Fix $\lambda\in\left]\frac{1}{9\sqrt{2}},1\right[$ and $\vrho\in \left]0,\frac{\lambda-\widetilde{\epsi}}{2}\right[$. Let $\seq{\xk}$ be the sequence generated by Algorithm~\ref{alg:MDBT} with the step-size $\gamma\equiv\frac{1-\kappa}{3+\widetilde{\epsi}+\vrho}$, where $\kappa \in ]0,1[$. Then the statements of Theorem~\ref{thm: Gauss} hold true with
	\[
	\nu=\frac{(1-\kappa)\Ppa{\lambda-\widetilde{\epsi}}}{3+\widetilde{\epsi}+\vrho}
	\qandq
	\varsigma=\frac{2\sqrt{2}\normm{\epsi}}{\sqrt{m\Ppa{c_s-\widetilde{\epsi}}}} .
	\]
\end{corollary} 

\begin{remark}\label{rem: gaussian}{\ }
	\begin{itemize}
		\item When the number of measurements is sufficiently large as in claim~\ref{thm: Gauss-i}, the SNR is large enough and the initial point $\xo$ is chosen randomly from a measure that has a density \wrt Lebesgue measure, then mirror descent converges, eventually linearly, to an element of $\Argmin(f)$ which is within a factor of the noise level from $\overline{\calX}$. The local convergence rate is dimension-independent. To the best of our knowledge, this is the first kind of results for the noisy phase retrieval problem.  
		\item When the number of measurements is in the less demanding regime of the second claim, then mirror descent with spectral initialization again converges to a noise region around $\overline{\calX}$. 
		\item We recover the rate of \cite[Theorem~3.2]{godeme2023provable} in the noiseless case.
		\item In the normalized setting of Corollary~\ref{cor:Gauss}, the convergence rate behaves as
		\[\Ppa{1-\nu}\leq \frac{2}{3}+O\Ppa{\pa{1-\lambda}+\kappa+\widetilde{\epsi}+\vrho}.\]
		\item It is important to observe that in the noisy case, the true vectors $\pm \avx$ are not even critical points of $f$. Nonetheless, Lemma~\ref{lem:interm} will show that $\pm \avx$ are actually $\frac{\normm{\epsi}^2}{m}$-minimizers.
	\end{itemize}
\end{remark}

\begin{proof}{\ }
	\begin{enumerate}
		
		\item[\ref{thm: Gauss-i}] 
		We prove this claim by  combining Theorem~\ref{thm:Det_reslt} and the characterization of the structure of $\crit(f)$ that we provide in Theorem~\ref{thm: Landscapethm}. For the moment, let us assume that the event $\calE_{\rm conv}$ holds true. 
		\begin{itemize}
			\item By construction, $\calE_{\rm conv}\subset\calE_{\rm inj}$ which means that the operator $A$ is injective showing the coercivity of the objective $f$ which implies that the sequence $\seq{\xk}$ generated by Algorithm~\ref{alg:MDBT} is bounded. 
			\item From the event $\calE_{\rm smad}$, we deduce that the function
			$f$ is $L$-smooth relative to $\psi$ with $L=3+\widetilde{\epsi}+\vrho\max\pa{\normm{\avx}^2/3+\norm{\epsi}{\infty},1}.$ Since the initializer is chosen at random with a fixed stepsize Theorem~\ref{thm:Det_reslt}-\ref{point-i}\ref{point-ii} guarantees that $\seq{\xk}$ converges to $\xsol\in\crit(f)\backslash\strisad(f)$ and $\seq{f(\xk)}$ also converges to $f(\xsol)$. 
			\item The event $\calE_{\rm scvx}$ shows that the function $f$ is $\sigma$-strongly convex relative to $\psi$ on $B(\overline{\calX},\rho)$ with $\sigma=\lambda\min\pa{\normm{\avx}^2,1}-\widetilde{\epsi}-\vrho\max\pa{\normm{\avx}^2/3+\norm{\epsi}{\infty},1}$. Given that Assumption~\ref{ass:SNR} holds, Corollary~\ref{cor: cond_r} implies that $\rho^2-\frac{4\normm{\epsi}^2}{m\sigma}>0$ where we recall that $\rho=\frac{(1-\lambda)\normm{\avx}}{\sqrt{3}}$. Let us denote $r^2 = \frac{\rho^2-\frac{4\normm{\epsi}^2}{m\sigma}}{{\max\pa{\Theta(\rho),1}}}$.
			\item Moreover, $\calE_{\rm strictsad}$ holds true, and thus $\crit(f)\backslash\strisad(f)=\Argmin(f)$, which means that for almost all initilizers $x_0$, 
			\begin{equation}\label{eq:conv_zer}
			\xk \to \xsol \in \Argmin(f) \qandq f(\xk)-\min f \to 0.        
			\end{equation}
		\end{itemize}
		We now claim that $\exists K>0,$ large enough, such that $\forall k\geq K, \dist(x_K,\overline{\calX}) \leq  r$ which will allow to invoke Theorem~\ref{thm:Det_reslt}\ref{point-iii}. By Lemma~\ref{lem:xbar_et_x}, we have for any $k\in\bbN$
		\begin{align*}
			\dist(\xk,\overline{\calX})&\leq\normm{\xk-\xsol}+8\frac{\normm{\epsi}}{\sqrt{m}\normm{\avx}} 
		\end{align*}
		with probability at least $1-e^{-\Omega(m)}$. Since $\xk\to\xsol$, there exists $K$ large enough such that $\forall k\geq K$
		\[
		\dist(x_k,\overline{\calX}) \leq 9\frac{\normm{\epsi}}{\sqrt{m}\normm{\avx}} 
		\]
		with the same probability. To conclude, it is then sufficient to show that
		\[
		9\frac{\normm{\epsi}}{\sqrt{m}\normm{\avx}} \leq r . 
		\]
		This is true for sufficiently high SNR, \ie under our Assumption~\ref{ass:SNR}.
		%
		Therefore we deduce from Theorem~\ref{thm:Det_reslt}-\ref{point-iii} that the sequence $(\xk)_{k\geq K}\in B(\overline{\calX},\rho)$  and for $k\geq K$,
		\begin{align*}
			\dist^2(\xk,\overline{\calX})&\leq \Ppa{1-\frac{(1-\kappa)\sigma}{L}}^{k-1}\rho^2+\frac{4\normm{\epsi}^2}{m\sigma},\\
			&\leq (1-\nu)^{k-K}\rho^2+4\frac{\normm{\epsi}^2}{m\sigma},\\
			&\leq (1-\nu)^{k-K}\rho^2+ \varsigma^2, 
		\end{align*}
		where we have used \eqref{eq:sigma_inf} \ie,  $\sigma\geq\frac{c_s\min\pa{\normm{\avx}^2,1}-\widetilde{\epsi}}{2}$ which implies that $\dist(\xk,\overline{\calX})\leq \varsigma.$ 
		
		Let us now compute the  probability that $\calE_{\rm conv}$ occurs. The events $\calE_{\rm smad}, \calE_{\rm scvx}$  are contained in $\calE_{\rm conH}$, see respectively Lemma~\ref{pro:Lsmad_G} and Lemma~\ref{pro:Localconvex_G}. From Lemma~\ref{lem:conhess_G},  $\calE_{\rm conH}$ occurs with a probability $1-5e^{-\zeta n}-4/n^2-2e^{-\frac{m\pa{\sqrt{1+\vrho}-1}^2}{8}}$ as soon as $m\geq C(\vrho)n\log(n)$. Besides, close observation of the Hessian concentration (noisy part) highlights that it implies the injectivity of the measurements thus $\calE_{\rm inj}$ is also contained in $\calE_{\rm conH}$. Thanks to Theorem~\ref{thm: Landscapethm}, the event $\calE_{\rm stricsad}$ holds with a probability $1-c/m$ as soon as $m\geq C(\vrho)n\log^3(n)$. Finally, we conclude with a union bound that $\calE_{\rm conv}$ holds with a
		probability  at least $1-2e^{-\frac{m\pa{\sqrt{1+\vrho}-1}^2}{8}}-5e^{-\zeta n}-4/n^2-c/m$ ($\zeta, c$ are a fixed numerical constant) for $m\geq C(\vrho)n\log^3(n)$, which complete the proof.  
		
		\item[\ref{thm: Gauss-ii}] By \cite[Lemma~B.4]{godeme2023provable}, the operator $A$ is injective entailing by the coercivity of $f$ that the sequence $\seq{\xk}$ is bounded. From Corollary~\ref{cor: cond_r}, when the signal-to-noise coefficient $c_s$ satisfies \eqref{eq:cond_r}, $r$ is well-defined.  
		Lemma~\ref{pro:spectralinit} shows that when $\vrho$ obeys \eqref{eq:spectralinitvrhobnd_nsy}, the initial point $\xo$ given by Algorithm~\ref{alg:algoSPIniit}  is in the right  $f-$attentive topology  at the distance at most $r=\sqrt{\rho^2-\frac{4\normm{\epsi}^2}{m\sigma}}$. Thanks to Lemma~\ref{pro:Localconvex_G}$, \rho$ is the radius of the ball $B(\overline{\calX},\rho)$ where we have $\sigma$-strong convexity relative to $\psi$ with $\sigma=\lambda\min\pa{\normm{\avx}^2,1}-\widetilde{\epsi}-\vrho\max\pa{\normm{\avx}^2/3+\norm{\epsi}{\infty},1}$. The last point to check before applying Theorem~\ref{thm:Det_reslt} comes from Lemma~\ref{pro:Lsmad_G} which shows that  $f$ is  $L-$smooth relative to $\psi$ with $L=3+\widetilde{\epsi}+\vrho\max\pa{\normm{\avx}^2/3+\norm{\epsi}{\infty},1}.$
		We deduce from Theorem~\ref{thm:Det_reslt} that
		\begin{align*}
			\dist^2(\xk,\overline{\calX})&\leq \Ppa{1-\frac{(1-\kappa)\sigma}{L}}^{k-1}\rho^2+\frac{4\normm{\epsi}^2}{m\sigma},\\
			&\leq (1-\nu)^{k-K}\rho^2+4\frac{\normm{\epsi}^2}{m\sigma}\\
			&\leq (1-\nu)^{k-K}\rho^2+\varsigma^2.
		\end{align*}
		Let us observe that this statement is true only on the intersection of the above events, we call it  $\calE'_{\rm conv}$. Let now compute the  probability that  $\calE'_{\rm conv}$ occurs.  We conclude with an appropriate union bound similar to the previous one, taking into account the fact that the  spectral initialization event is contained in $\calE_{\rm conH}$. Finally, the statement holds with a probability  at least $1-2e^{-\frac{m\pa{\sqrt{1+\vrho}-1}^2}{8}}-5e^{-\zeta n}-4/n^2$ ($\zeta$ is a fixed numerical constant) for $m\geq C(\vrho)n\log(n)$, which completes the proof.   
	\end{enumerate}
\end{proof}

\section{Numerical experiments}\label{sec:mdnsy_numexp}
In this section, we discuss some experiments to illustrate and validate numerically the efficiency of our phase recovery algorithm. In each instance, we measured the relative error between the reconstructed vector $\tilde{x}$ and the true signal one $\avx$ as
\begin{align}\label{relative_error_nsy}
	\frac{\dist(\widetilde{x},\overline{\calX})}{\normm{\avx}}.
\end{align}
In the experiments, we set $\normm{\avx}=1$ and $\widetilde{x}$ was the output of Algorithm~\ref{alg:MDBT} at iteration $K$ large enough.

\subsection{Experiments with Gaussian sensing vectors}

\subsubsection{Reconstruction of 1D signals}\label{subsubsec:1Dgaussny}
The aim is to reconstruct a randomly generated one-dimensional signal of length $n = 128$ from $m$ noisy observations where the sensing vectors were drawn \iid from the standard Gaussian ensemble. The noise vector $\epsi$ is chosen uniform such that $\widetilde{\epsi}=10^{-5}$. 

In Figure~\ref{fig: reconstruction_gauss}, the blue line shows the evolution of the objective (left) and the relative error (right) using Algorithm~\ref{alg:MDBT} with $m = 128\times\log^2(128)$, where the initial point was drawn from the uniform distribution. The red line is the result of reconstruction from $m=5\times128\times\log(128)$ measurements where Algorithm~\ref{alg:MDBT} was intialized using the spectral method in Algorithm~\ref{alg:algoSPIniit} (the top eigenvalue was computed with the power iteration method using $200$ iterations). In both cases, we run Algorithm~\ref{alg:MDBT} with a constant step-size $\gamma=\frac{0.99}{3+10^{-5}}$ (see Theorem~\ref{thm: Gauss}). As predicted by the latter, the curves in blue (\ie with random initialization) have two regimes: a sublinear regime and then a local linear regime. Moreover, with spectral initialization (red curves), $f$ and the relative error converge linearly at the same rate as in the local regime of the blue curves, hence confirming our theoretical findings. As anticipated also, both $f$ and the relative error eventually stabilize at a plateau whose level is governed by the noise level.

\begin{figure}[htbp]
	\centering
	\includegraphics[trim={0cm 8cm 0 8cm},clip,width=0.5\linewidth]{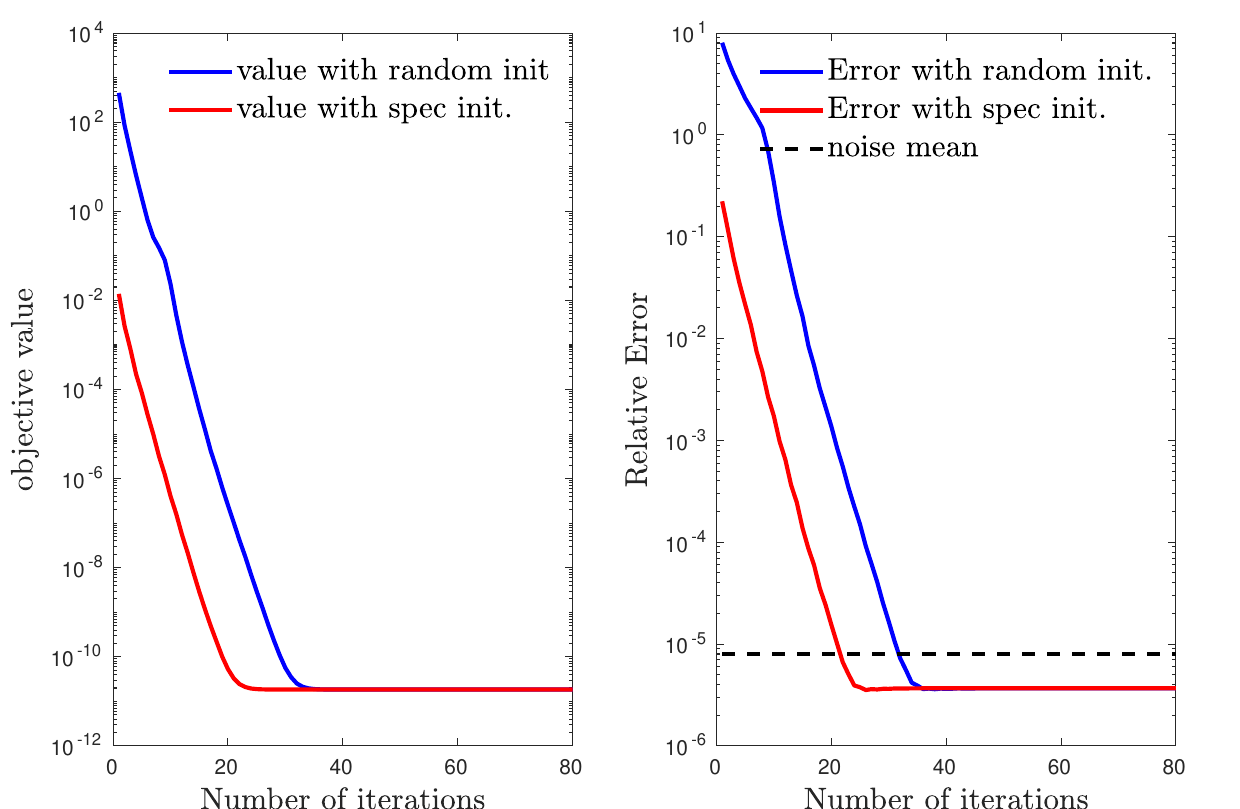}
	\caption{Reconstruction of signal from  Gaussian measurements. The noise mean is $\widetilde{\epsi}$.}
	\label{fig: reconstruction_gauss}
\end{figure}


\subsubsection{Comparison with the Wirtinger flow}
We compared our mirror descent algorithm (with and without spectral initialization) with the Wirtinger flow \cite{Candes_WF_2015}. However, the Polyak subgradient method proposed in \cite{davis_nonsmooth_2020}, that we included in our comparison in \cite{godeme2023provable}, is only applicable to the noiseless case as it needs the value of $\min f$ which is no longer known in presence of noise. We used the spectral method in Algorithm~\ref{alg:algoSPIniit} for the Wirtinger flow, and we compared with mirror descent with and without spectral initialization. 

For this, we report the results of an experiment designed to estimate the phase retrieval probability of success of each algorithm as we vary $n$ and $m$. The results are depicted in Figure~\ref{Phase_transition_Gaussian}. For each pair $(n,m)$, we generated $100$ instances and solved them with each algorithm. Each diagram shows the empirical success probability (among the $100$ instances) of the corresponding algorithm. An algorithm is declared as successful if the relative error \eqref{relative_error_nsy} is less than $\frac{2\normm{\epsi}}{\sqrt{m\sigma}}\approx 10^{-5}$. The grayscale of each point in the diagrams reflects the observed probability of success, from $0\%$ (black) to $100\%$ (white). The solid curve marks the prediction of the phase transition edge. On the left panel of Figure~\ref{Phase_transition_Gaussian}, we also plot a profile of the phase diagram extracted at $n=128$.

One observes a phase transition phenomenon that is in agreement with the predicted sample complexity bound shown as a solid line. Moreover, mirror descent performs better than the Wirtinger flow with both use spectral intialization. Mirror descent with uniform random initialization has a weaker recovery performance with a transition to success occurring at a higher threshold compared to the version of mirror descent with spectral initialization. This is in agreement with our theoretical findings as more measurements are needed in this case to ensure stable recovery.

\begin{figure}
	\centering
	\subfloat{
		\centering
		\includegraphics[trim={0cm 5cm 0 5cm},clip,width=0.4\textwidth]{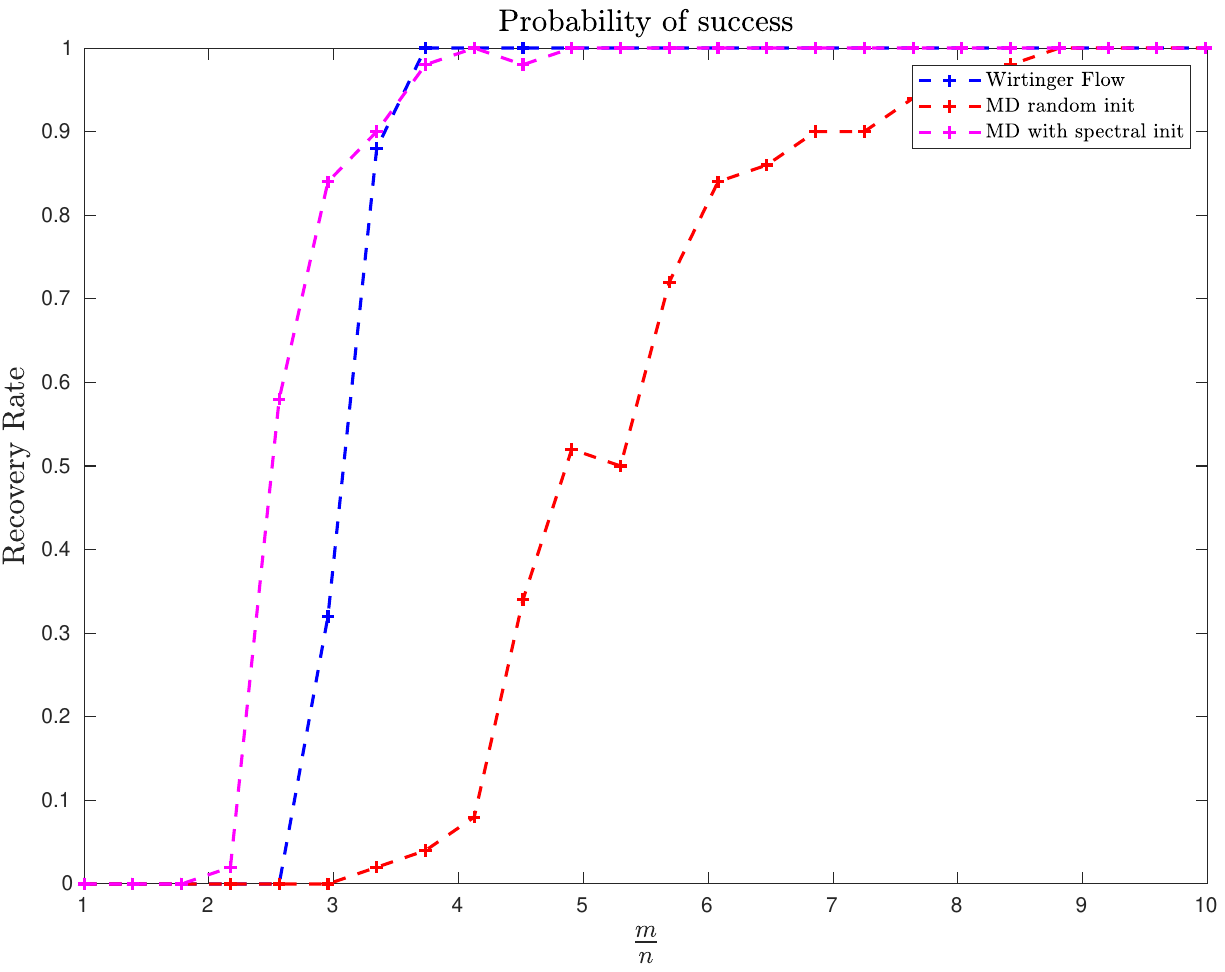}}\hspace{-0.1cm}\quad
	\subfloat{
		\centering
		\includegraphics[trim={0cm 5cm 0 5cm},clip,width=0.5\textwidth]{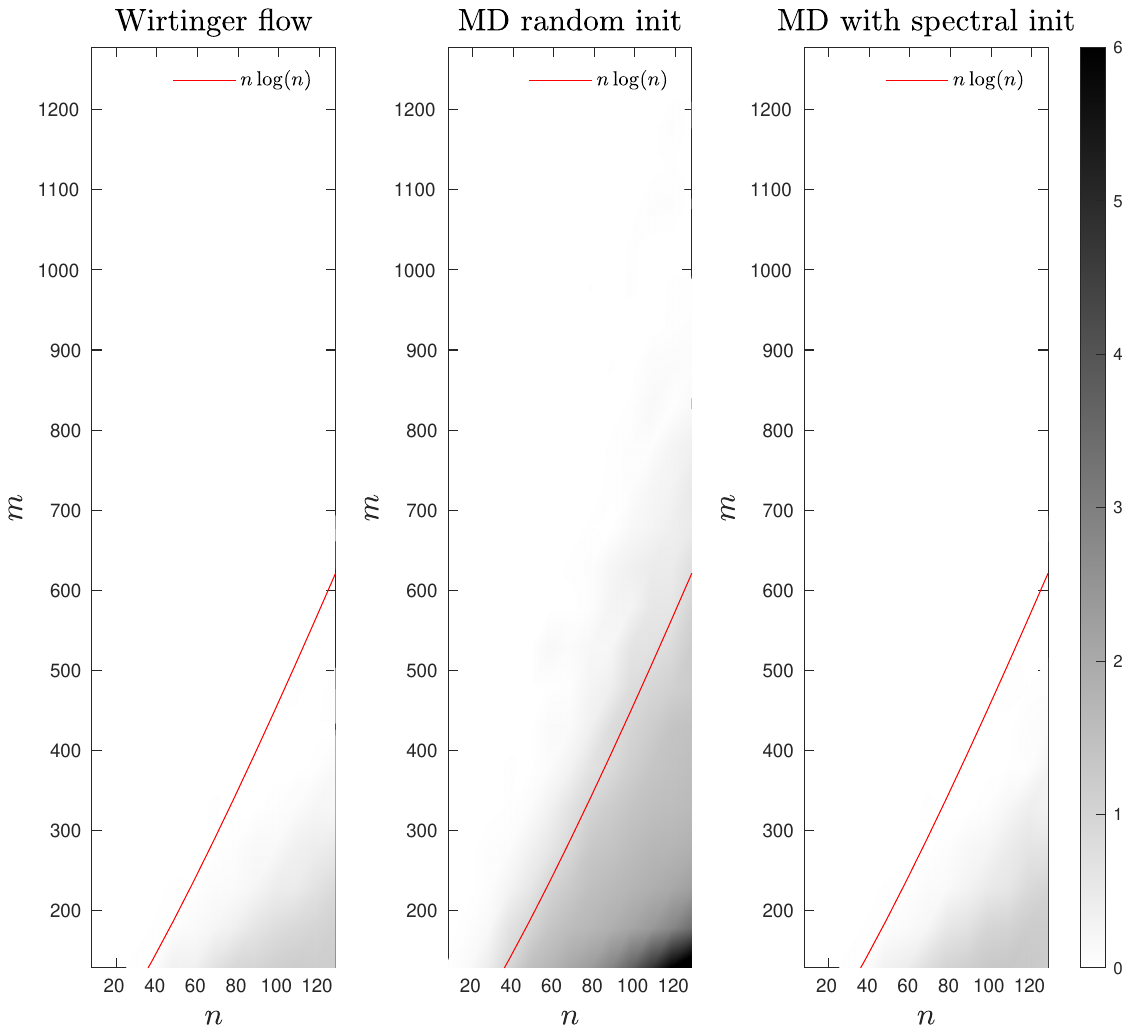}}
	\caption{Phase diagrams for Gaussian measurements.}
	\label{Phase_transition_Gaussian}
\end{figure}

\subsection{Experiments with the CDP model}

We now turn to the case of structured measurements from the CDP model. This model uses $P$ coded diffraction patterns/masks followed by a Fourier transform. The observation model is given by 
\begin{eqnarray}\label{eq:PRCDP}
y=\Ppa{\left|\sum_{\ell=0}^{n-1}\avx_{\ell}d_p[\ell] e^{-i\frac{2\pi j\ell}{n}}\right|^2+\epsi_{j,p}}_{j,p},
\end{eqnarray}
where $j\in\{0,\ldots,n-1\}$ and $p\in\{0,\ldots,P-1\}$, $\epsi$ is the noise. The total number of measurements is thus $m=nP$ (\ie the oversampling factor is $P$). $(d_p)_{p \in [P]}$ are \iid copies of a random variable $d$, and in our experiment, $d$ takes values in $\left\{-1,0,1 \right\}$ with probability $\left\{1/4,1/2,1/4 \right\}$. Here we performed a similar experience to the Gaussian case described in Section~\ref{subsubsec:1Dgaussny}, where we chose the number of masks $P=7\times\log^3(128)$ and a constant step-size $\gamma=\frac{0.99}{2+\widetilde{\epsi}}$, with $\widetilde{\epsi}=10^{-5}$. Despite the lack of theoretical guarantees for the CDP model in the noisy case, that we conjecture are true, one can observe in Figure~\ref{fig: reconstruction_cdp} that we have very similar results to those for Gaussian measurements. 

\begin{figure}[htbp]
	\centering
	\includegraphics[trim={0cm 8cm 0 8cm},clip,width=0.6\linewidth]{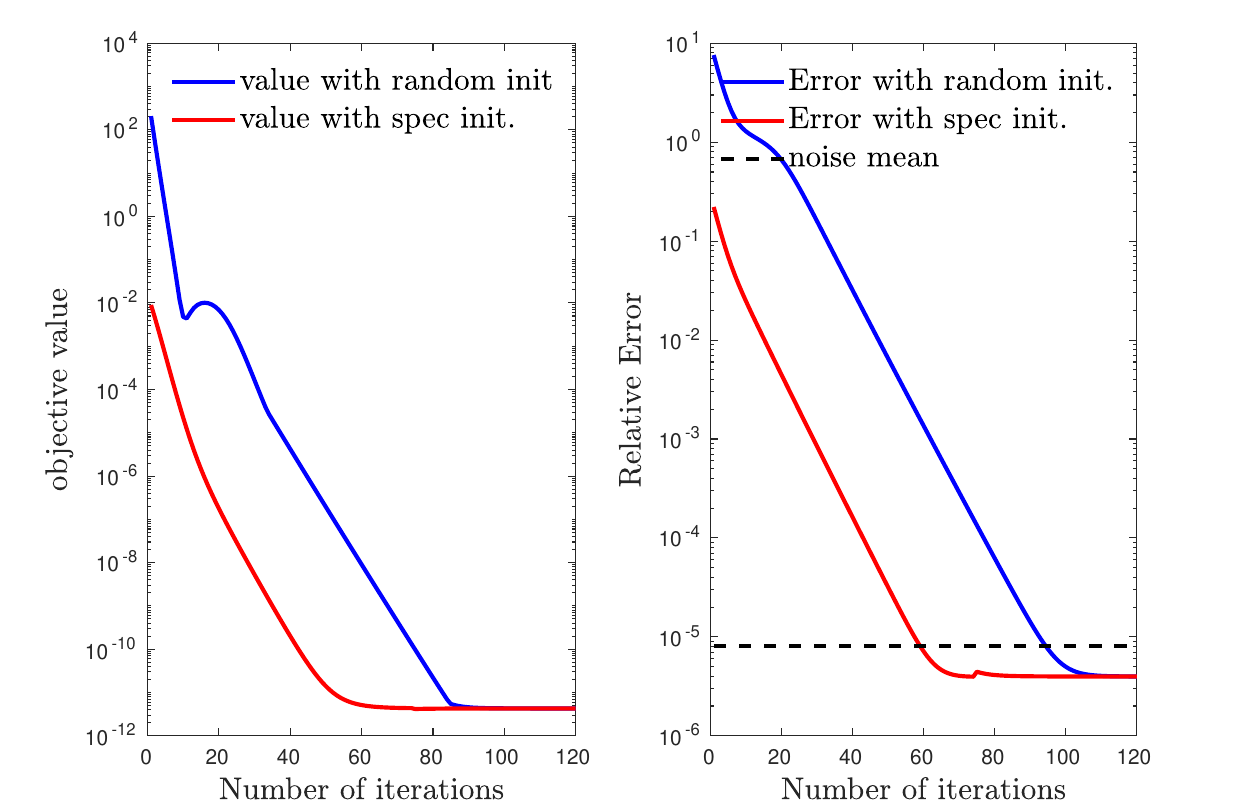}
	\caption{Reconstruction of signal from  Noisy CDP. The noise mean is $\widetilde{\epsi}$}
	\label{fig: reconstruction_cdp}
\end{figure}


\subsection{Recovery of a 2D image}
In this experiment, we work with the image of the beautiful Unicaen's\footnote{Unicaen = University of Caen} phoenix whose dimension is $396\times 396$. Our goal is to recover the image from noisy CDP measurements with $P=90$ masks. The noise is chosen such that $\widetilde{\epsi}=10^{-5}$. We used the spectral method to find the initial guess and run mirror descent for $1000$ iterations.  The results are displayed in Figure~\ref{reconstruction2D_CDP_nsy} showing that our algorithm converges to the desired image with a relative error of order $10^{-2}$.

\begin{figure}
	\centering
	\subfloat[Original Unicaen's \tcb{phoenix}]{ 
		\centering
		\includegraphics[trim={2cm 10cm 0 10cm},clip,width=0.5\linewidth]{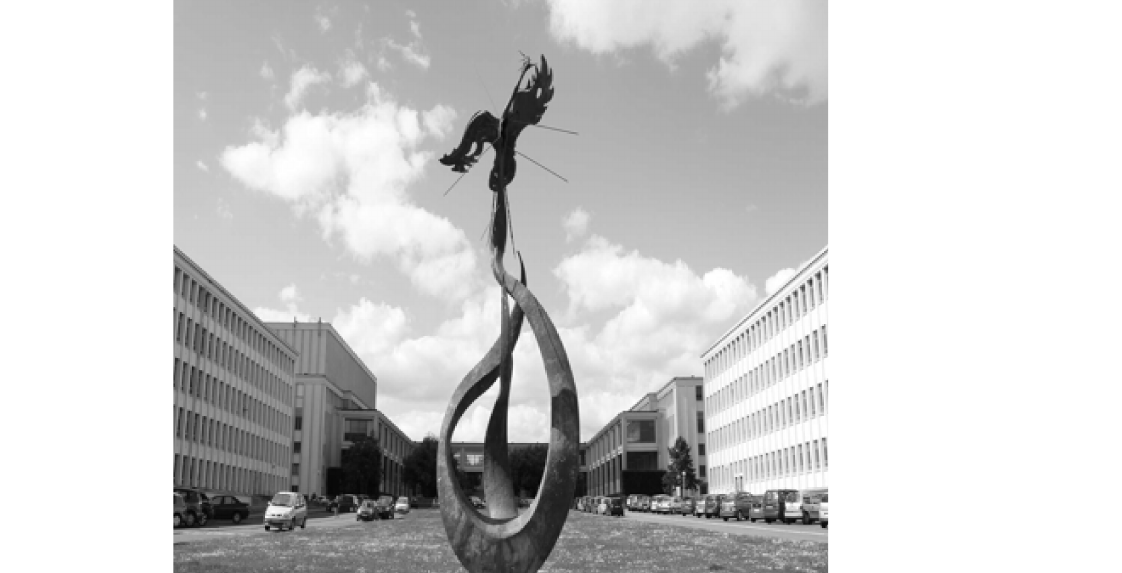}}\hspace{-1.8cm}\quad
	\subfloat[The CDP measurements averaged over the $P=90$ masks]{
		\centering
		\includegraphics[trim={2cm 10cm 0 10cm},clip,width=0.5\linewidth]{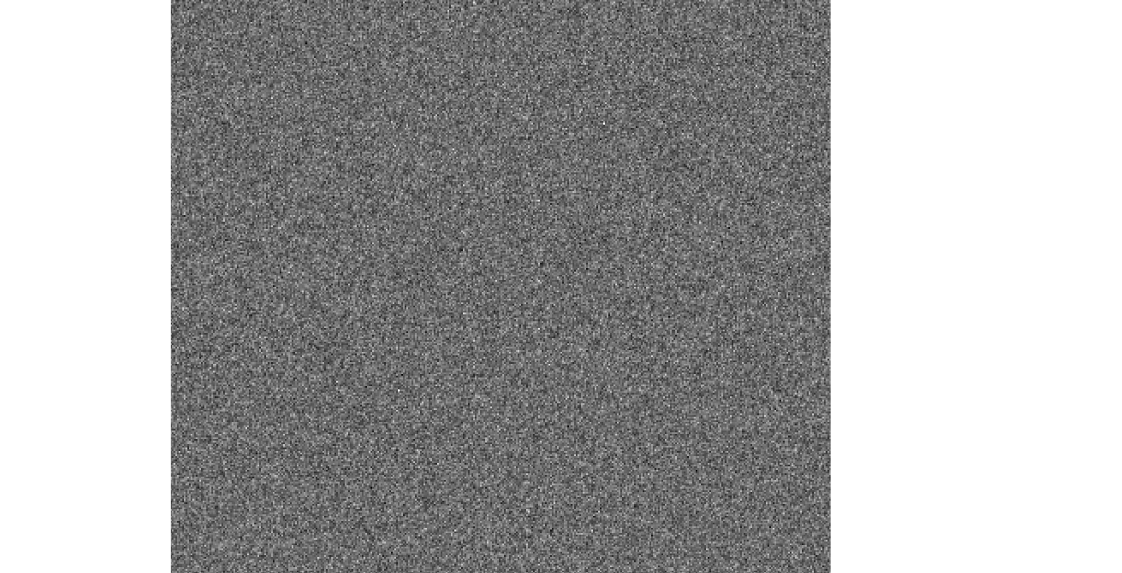}}\\
	\subfloat[Recovered Unicaen's \tcb{phoenix}]{
		\centering
		\includegraphics[trim={2cm 10cm 0 10cm},clip,width=0.5\linewidth]{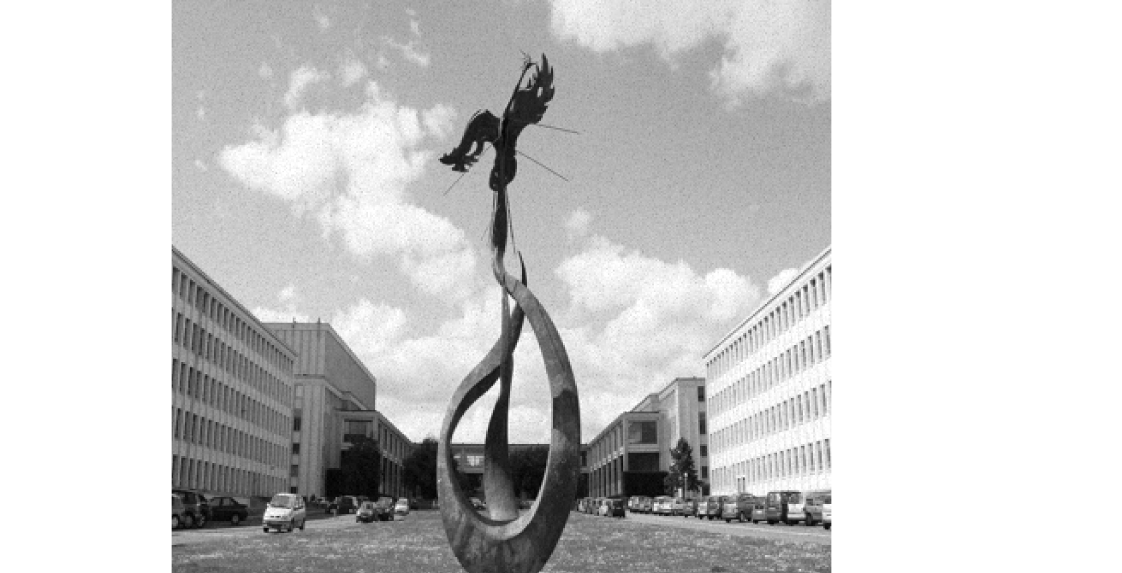}}    
	\caption{Reconstruction of an image from noisy CDP measurements.}
	\label{reconstruction2D_CDP_nsy}
\end{figure}

\begin{appendices}\label{sec:mdbtny_append}
\section{Proofs for the Deterministic Case}\label{sec:proofs_mdbtny_deter}
Throughout the work, we  use when it is convenient the  following decomposition of the objective function $f$ in \eqref{eq:formulepro}.
\begin{equation}\label{eq:defsum}
\forall x\in\bbR^n,\quad f(x)=f_{\mathrm{NL}}(x)+f_{\mathrm{Ny}}(x),
\end{equation}
where $f_{\mathrm{NL}}$ and $f_{\mathrm{Ny}}$ denote respectively the noiseless and the noisy part of $f$ and we have  
\begin{align}
	f_{\mathrm{NL}}(x)= \qsom{\loss{|\transp{a_r}x|^2}{|\transp{a_r}\avx|^2}},\qquad
	f_{\mathrm{Ny}}(x)=-\hsom{\epsi[r]\para{|\transp{a_r}x|^2-|\transp{a_r}\avx|^2}}+\frac{\normm{\epsi}^2}{4m}.
\end{align}
The following computations are straightforward:
\begin{align}\label{eq:derivpsi}
	\nabla\psi(x)&=\Ppa{\normm{x}^2+1}x, &\nabla^2\psi(x)&=\Ppa{\normm{x}^2+1}\Id+2x\transp{x}, 
\end{align}
\begin{align}\label{eq:gradf}
	\nabla f(x)= \som{\Ppa{|\transp{a_r}x|^2-|\transp{a_r}\avx|^2}a_r\transp{a_r}x}-\som{\epsi[r]a_r\transp{a_r}x} 
\end{align}
and
\begin{align}\label{eq:hessf}
	\nabla^2 f(x)= \som{\Ppa{3|\transp{a_r}x|^2-|\transp{a_r}\avx|^2}a_r\transp{a_r}}-\som{\epsi[r]a_r\transp{a_r}}.
\end{align}

\subsection{Proof of Lemma~\ref{Tsmad}}\label{PrTsmad}
\begin{proof}
	For all $x,u \in \bbR^n$, it easy to check that 
	\begin{align*}
		\pscal{u,\nabla^2\psi(x)u}\geq \para{\normm{x}^2+1}\normm{u}^2.
	\end{align*}
	On the other hand, we have 
	\begin{align*}
		\pscal{u,\nabla^2f(x)u} 
		&= \som{\para{3|\transp{a_r}x|^2-|\transp{a_r}\avx|^2-\epsi[r]}|\transp{a_r}u|^2}\\
		&\leq \som{\para{3|\transp{a_r}x|^2-\epsi[r]}|\transp{a_r}u|^2} \\
		&\leq \para{\som{\para{3\normm{a_r}^2+\norm{\epsi}{\infty}}\normm{a_r}^2}}\Ppa{\normm{x}^2+1}\normm{u}^2.
	\end{align*}
	Thus for any $L\geq \frac{1}{m}\sum\limits_{r=1}^m\normm{a_r}^2\para{3\normm{a_r}^2+\norm{\epsi}{\infty}}$, we have for all $x \in \bbR^n$
	\begin{equation}\label{hessmat}
	\nabla^2f(x) \lon L\nabla^2\psi(x). 
	\end{equation}
	The claim then follows from Lemma~\ref{lem:bregcomp} with $g=f$ and $\phi=L\psi$, and Proposition~\ref{pro:bregman}\ref{pro:bregman2}. 
\end{proof}

\subsection{Proof of Theorem~\ref{thm:Det_reslt}}\label{PrMaintheo}
Let us start with the following intermediate results.
\begin{lemma}\label{Alllinone}
	Let the sequence $(\xk)_{k \in \N}$ be generated by Algorithm~\ref{alg:MDBT}. Then  for all $u \in \bbR^n$,
	\begin{equation}\label{eq:Allinone}
	D_{\psi}\pa{u,\xkp} + \gak \pa{f(\xkp)-\min f}\leq D_{\psi}\pa{u,\xk}-\kappa D_{\psi}\pa{\xkp,\xk}-\gak D_f\Ppa{u,\xk}+ \gak\pa{f(u)-\min f}. 
	\end{equation}
\end{lemma}

\begin{proof} 
	From \cite[Lemma~A.2]{godeme2023provable}, we have
	\begin{align*}
		\forall u\in\bbR^n,\quad D_{\psi}\Ppa{u,\xkp} + \gamma_k\Ppa{f(\xkp)-f(u)}\leq D_{\psi}\Ppa{u,\xk}-\kappa D_{\psi}\Ppa{\xkp,\xk}-\gamma_k D_f\Ppa{u,\xk} .
	\end{align*}
	Subtracting $\min f$ from both sides yields the result.
\end{proof}

\begin{lemma}\label{lem:interm} 
	Assume that $\Argmin(f) \neq \emptyset$. We have 
	\[
	0 \leq f(\pm\avx)-\min f\leq \frac{\normm{\epsi}^2}{m}.
	\] 
\end{lemma}
\begin{proof} 
	Let $\xsol\in\Argmin(f)$. We prove the claim for $\avx$. By optimality, we have $f(\xsol)\leq f(\avx)=\frac{\normm{\epsilon}^2}{4m}$, which equivalently reads
	\begin{equation*}
		\Som{\Ppa{|\transp{a_r}\xsol|^2-|\transp{a_r}\avx|^2}}^2 \leq 2\Som{\epsilon_r\Ppa{|\transp{a_r}\xsol|^2-|\transp{a_r}\avx|^2}} .
	\end{equation*}
	Applying Young's inequality to the right-hand side then entails
	\begin{equation}\label{eq:bndobj1}
	(1-\delta)\Som{\Ppa{|\transp{a_r}\xsol|^2-|\transp{a_r}\avx|^2}}^2 \leq \frac{\normm{\epsilon}^2}{\delta} \quad \forall \delta \in ]0,1[ .
	\end{equation}
	Consequently, using \eqref{eq:bndobj1} and Young's inequality again, we get
	\begin{align*}
		4m D_f\pa{\pm\avx,\xsol} 
		&=  4m\Pa{f(\avx)-f(\xsol)} \\
		&=  2\Som{\epsi[r]\Ppa{|\transp{a_r}\xsol|^2-|\transp{a_r}\avx|^2}}-\Som{\Ppa{|\transp{a_r}\xsol|^2-|\transp{a_r}\avx|^2}^2}\\
		&\leq  2\Som{\epsi[r]\Ppa{|\transp{a_r}\xsol|^2-|\transp{a_r}\avx|^2}}\\
		&\leq  \frac{\normm{\epsi}^2}{1-\delta} + (1-\delta)\Som{\Ppa{|\transp{a_r}\xsol|^2-|\transp{a_r}\avx|^2}^2}
		\leq  \frac{\normm{\epsi}^2}{\delta(1-\delta)} .
	\end{align*}
	The minimal value of the right hand side is $4\normm{\epsi}^2$ attained for $\delta=1/2$. 
\end{proof}

\begin{proof} $~$
	
	\paragraph*{\ref{point-i}-\ref{point-ii}} Similar to the proofs of the corresponding claims in \cite[Theorem~2.11]{godeme2023provable}.

	\paragraph*{\ref{point-iii}}
	We give the proof for $\avx$ and obviously the same holds at $-\avx$. We proceed by induction. We first have that $\xo \in B(\avx,r) \subset B(\avx,\rho)$ since $r \leq \rho$. Suppose now that for $k \geq 0$, $(x_i)_{0 \leq i \leq k} \subset B(\avx,\rho)$. Applying Lemma~\ref{Alllinone} at $\avx$ and using Lemma~\ref{lem:interm}, we have
	\begin{align}
		D_{\psi}(\avx,\xkp)
		&\leq D_{\psi}(\avx,\xk)-\kappa D_{\psi}(\xkp,\xk)-\gamma D_f\pa{\avx,\xk} + \gamma\frac{\normm{\epsi}^2}{m} \label{eq:Dpsiloclinear}\\  
		&\leq D_{\psi}(\avx,\xk)-\gamma D_f\para{\avx,\xk}+\gamma\frac{\normm{\epsi}^2}{m} \nonumber\\
		&\leq (1-\gamma\sigma)D_{\psi}(\avx,\xk)+\gamma\frac{\normm{\epsi}^2}{m} \nonumber ,
	\end{align}
	where we also used positivity of $D_{\psi}$ and local $\sigma$-relative strong convexity of $f$ since $\xk \in B(\avx,\rho)$. Iterating the last inequality, we get
	\begin{align}
		D_{\psi}(\avx,\xkp)
		&\leq (1-\gamma\sigma)^{k+1} D_\psi(\avx,\xo) + \gamma\frac{\normm{\epsi}^2}{m}\sum_{i=0}^k(1-\gamma\sigma)^i \nonumber \\
		&\leq (1-\gamma\sigma)^{k+1} D_\psi(\avx,\xo) + \frac{\normm{\epsi}^2}{m\sigma}\Ppa{1-(1-\gamma\sigma)^{k+1}} \label{eq:Dpsiloclinearnoisy} \\
		&\leq D_\psi(\avx,\xo) + \frac{\normm{\epsi}^2}{m\sigma} . \nonumber
	\end{align}
	It then follows from Proposition~\ref{pro:bregman}\ref{pro:bregman4} that
	\begin{align*}
		\normm{\xkp-\avx}^2
		\leq \normm{\xo-\avx}^2\Theta\pa{\rho}+\frac{2\normm{\epsi}^2}{m\sigma} 
		\leq r^2\Theta\pa{\rho}+\frac{2\normm{\epsi}^2}{m\sigma} 
		\leq \rho^2 .
	\end{align*}
	This shows \eqref{eq:loclinearnoisy}.
\end{proof}

\section{Proofs for Gaussian Measurements}\label{sec:proofs_mdbtny_rand}

\subsection{Expectation and deviation of the hessian}
\begin{lemma}\label{lem:expe_lem_G}\textbf{(Expectation of the hessian)}{\ }
	If the sensing vectors $(a_r)_{r\in[m]}$ are sampled following the Gaussian model then we have for any $x\in\bbR^n$,  
	\begin{align}\label{eq:mdnsy-expe_lem_G}
		\esp{\nabla^2 f(x)}=3\para{2x\transp{x}+ \normm{x}^2\Id}-2\avx\transp{\avx}-\normm{\avx}^2\Id-\widetilde{\epsi}\Id. 
	\end{align}
\end{lemma}
\begin{proof}
	The proof combines \eqref{eq:hessf},\cite[Lemma~B.1]{godeme2023provable} for the expectation of the first (\ie noiseless part), and the last term comes the fact that the sensing vectors have zero mean and unit covariance.
\end{proof}

\begin{lemma}\label{lem:conhess_G}\textbf{(Concentration of the hessian)}
	Fix  $\vrho\in]0,1[$, if the number of samples obeys $m\geq C(\vrho)n\log n$, for some sufficiently large constant $C(\vrho)>0$  then 
	\begin{align}\label{eq:conhess_G}
		\normm{\nabla^2f(x)-\esp{\nabla^2f(x)}}\leq \vrho\Ppa{\normm{x}^2+\frac{\normm{\avx}^2}{3}+\norm{\epsi}{\infty}}
	\end{align}
	holds simultaneously  for all $x\in\bbR^n$ with a probability at least $1-5e^{-\zeta n}-\frac{4}{n^2}-2e^{-\frac{m\pa{\sqrt{1+\vrho}-1}^2}{8}},$  where $\zeta$ is a fixed numerical constant. 
\end{lemma}

\begin{proof}
	By the triangle inequality, we have 
	\begin{align*}
		\normm{\nabla^2f(x)-\esp{\nabla^2f(x)}}\leq &\normm{\som{\Ppa{3|\transp{a_r}x|^2a_r\transp{a_r}-|\transp{a_r}\avx|^2a_r\transp{a_r}}}-\Ppa{6x\transp{x}+3\normm{x}^2\Id-2\avx\transp{\avx}-\normm{\avx}^2\Id}}\\
		&+\normm{\som{\epsi[r]a_r\transp{a_r}-\widetilde{\epsi}\Id}}.
	\end{align*}
	The concentration of the first term has been proved for the noiseless case (see \cite[Lemma B.3]{godeme2023provable}) with a probability $1-5e^{-\zeta n}-\frac{4}{n^2}$. For the noisy part, we have
	\begin{align}\label{eq:2ndconc}
		\normm{\som{\epsi[r]a_r\transp{a_r}-\widetilde{\epsi}\Id}}\leq \frac{\norm{\epsi}{\infty}}{m}\normm{\Som{\para{a_r\transp{a_r}-\Id}}}=\frac{\norm{\epsi}{\infty}}{m}\normm{\transp{A}A-\Id},
	\end{align}
	where $A$  is the $m\times n$matrix whose $r-$th row is the vector $\transp{a_r}$. 
	From \cite[Lemma~B.4]{godeme2023provable}, we get that for any $\vrho\in ]0,1[$,   
	\begin{align*}
		\normm{\som{\epsi[r]a_r\transp{a_r}-\widetilde{\epsi}\Id}}\leq\vrho\norm{\epsi}{\infty}
	\end{align*}
	with a probability at least $1-2e^{-mt^2/2}$, with $\frac{\vrho}{4}=t^2+t$ with $m\geq\frac{16}{\vrho^2}n$. 
	We conclude by applying a simple union bound.  
\end{proof}

\subsection{Optimal solution near the true vector}\label{Pr_xbar_et_x}
\begin{lemma}\label{lem:xbar_et_x}{\ }
	Assume that Assumption~\ref{ass:SNR} holds and that $m \geq c n$ where $c$ is a positive numerical constant. Then for any $\xsol \in \Argmin(f)$, 
	\begin{align}
		\dist(\xsol,\overline{\calX})\leq 8 \frac{\normm{\epsi}}{\sqrt{m}\normm{\avx}} 
	\end{align}
	holds with probability $1-e^{-\Omega(m)}$.
\end{lemma}
\begin{proof}
	Let us use the following notation: $X^\star=\xsol\transp{\xsol},\overline{X}=\avx\transp{\avx}, \overline{\epsi}\eqdef\frac{\normm{\epsi}}{\sqrt{m}}$ and $\epsi_0 \eqdef 4\overline{\epsi}$.
	
	By optimality of $\xsol$, we have $f(\xsol)\leq f(\avx)=\frac{\normm{\epsilon}^2}{m}$, which also implies that
	\begin{equation*}
		\Som{\Pa{|\transp{a_r}\xsol|^2-|\transp{a_r}\avx|^2}}^2 \leq 2\Som{\epsilon_r\Pa{|\transp{a_r}\xsol|^2-|\transp{a_r}\avx|^2}} .
	\end{equation*}
	Applying Young's inequality to the right-hand side then entails
	\begin{equation}\label{eq:bndobj}
	\Som{\Pa{|\transp{a_r}\xsol|^2-|\transp{a_r}\avx|^2}}^2 \leq 4\normm{\epsilon}^2 .
	\end{equation}

	Fix $\zeta \in ]0,1[$. Using \cite[Lemma~1]{chen_solving_2017}, there are positive numerical constants $C$ and $C'$ such that if $m \gtrsim n \zeta^{-2}\log(1/\zeta)$, then with probability at least $1 - C' e^{-C\zeta^2m}$, we have
	\[
	\som{\Pa{|\transp{a_r}\xsol|^2-|\transp{a_r}\avx|^2}}^2 \geq 0.81 (1-\zeta)^2 \normF{X^\star-\overline{X}}^2 .
	\]
	Thus, in view of \eqref{eq:bndobj}, we the same probability, we have
	\begin{equation}\label{eq:isorank2}
	\normF{X^\star-\overline{X}} \leq \frac{20}{9 (1-\zeta)}\overline{\epsi}.
	\end{equation}
	Therefore taking $\zeta=0.4$ in \eqref{eq:isorank2} one has
	\begin{equation}\label{eq:bound-frob}
	\normF{X^\star-\overline{X}} \leq \epsi_0.
	\end{equation}
	The rest of the proof is inspired by that of \cite[Theorem~1.2]{candes_phaselift_2013}. Since $\normm{\xsol}^2$ and $\normm{\avx}^2$ are the largest eigenvalues of the rank-one symmetric matrices $X^\star$ and $\overline{X}$, we have from Weyl's perturbation inequality of the eigenvalues that 
	\begin{equation*}
		\abs{\normm{\xsol}^2-\normm{\avx}^2}\leq \normF{X^\star-\overline{X}} \leq \epsi_0 .
	\end{equation*}
	Let us assume that $\normm{\avx}^2=1$ and the general case is obtained via a simple rescaling argument. Under Assumption~\ref{ass:SNR}, $\overline{\epsi}$ is small enough so that $\epsi_0 < 1$. We then get that $\normm{\xsol}^2\in[1-\epsi_0,1+\epsi_0]$. The $\sin$-$\theta$-Theorem \cite{DavisKahan70} implies that 
	\[
	\abs{\sin\theta} \leq \frac{\normF{X^\star-\overline{X}}}{\normm{\xsol}^2} \leq \frac{\epsi_0}{1-\epsi_0},
	\]
	where $0\leq\theta\leq\frac{\pi}{2}$ is the angle between $x^\star$ and $\avx$ which are the eigenvectors of $X^\star$ and $\overline{X}$ associated to the eigenvalues $\normm{\xsol}^2$ and $1$, respectively. We can then write
	\[
	\xsol=\normm{\xsol}(\cos\theta \avx + \sin\theta\orth{\avx}) ,
	\]  
	where $\orth{\avx}$ is a unit vector orthogonal to $\avx$. We apply the Ihâmessou-Pythagoras theorem to get
	\[
	\normm{\xsol-\avx}^2=\Ppa{1-\normm{\xsol}\cos\theta}^2+ \normm{\xsol}^2\sin^2\theta. 
	\]	
	Since $\cos\theta=\sqrt{1-\sin^2\theta}$, we have  
	\[
	1+\epsi_0 \geq \sqrt{1+\epsi_0} \geq \normm{\xsol}\cos\theta\geq \sqrt{1-\epsi_0-\frac{\epsi_0^2}{1-\epsi_0}}\geq 1-\epsi_0 ,
	\]	
	where we used that $\epsi_0<1/3$ in the last inequality. We then get that
	\[
	\Ppa{1-\normm{\xsol}\cos\theta}^2 \leq \epsi_0^2 .
	\]
	In turn
	\[
	\normm{\xsol-\avx}^2\leq \epsi_0^2+\frac{\epsi_0^2(1+\epsi_0)}{(1-\epsi_0)^2}\leq 4\epsi_0^2
	\]	
	for $\epsi_0<1/3$.
	We also know that 
	\[
	\normm{\xsol-\avx} \leq 2 + \epsi_0 \leq 7/3 
	\]
	for $\epsi_0<1/3$. We therefore get that
	\[
	\dist(\xsol,\overline{\calX})\leq 8\min(\overline{\epsi},1) \leq  8 \overline{\epsi} ,
	\]	
	where the last inequality is a consequence of Assumption~\ref{ass:SNR}.
\end{proof}

\subsection{Relative smoothness}
\begin{lemma}\label{pro:Lsmad_G} Fix $\vrho\in]0,1[$, if the event $\calE_{\rm conH}$ holds true then,
	\begin{align}\label{eq:pro:Lsmad_G}
		\forall x,z \in\bbR^n, \quad  D_f(x,z)\leq \para{3+\widetilde{\epsi}+\vrho\max\para{\normm{\avx}^2/3+\norm{\epsi}{\infty},1}}D_{\psi}(x,z).
	\end{align}
\end{lemma}

\begin{proof}
	Let fix $\vrho\in]0,1[$, for any  $u\in\bbR^n$, we have
	\begin{align*}
		\nabla^2f(u)&\lon \esp{\nabla^2f(u)}+\vrho\pa{\normm{u}^2+\normm{\avx}^2/3+\norm{\epsi}{\infty}}\Id,\\
		&\lon 3\Ppa{2u\transp{u}+\normm{u}^2\Id} -2\avx\transp{\avx}-\normm{\avx}^2\Id-\widetilde{\epsi}\Id+\vrho\max\pa{\normm{\avx}^2/3+\norm{\epsi}{\infty},1}\para{\normm{u}^2+1}\Id,\\
		&\lon 3\Ppa{2u\transp{u}+\pa{\normm{u}^2+1}\Id}+\widetilde{\epsi}\Id+\vrho\max\pa{\normm{\avx}^2/3+\norm{\epsi}{\infty},1}\nabla^2\psi(x),\\
		&\lon 3\nabla^2\psi(u)+\widetilde{\epsi}\nabla^2\psi(u)+\vrho\max\pa{\normm{\avx}^2/3+\norm{\epsi}{\infty},1}\nabla^2\psi(x),\\
		&\lon  \para{3+\widetilde{\epsi}+\vrho\max\pa{\normm{\avx}^2/3+\norm{\epsi}{\infty},1}}\nabla^2\psi(u),\label{taylo}\numberthis
	\end{align*}
	We conclude by applying  Lemma~\ref{lem:bregcomp} in the segment $[x,z]$.
\end{proof}

\subsection{Local relative strong convexity}

\begin{lemma}\label{pro:Localconvex_G}{\ }  Fix  $\lambda\in\left]\frac{1}{9\sqrt{2}},1\right[$ and  for  $\vrho\in\left]0,\frac{\lambda\min\pa{\normm{\avx}^2,1}-\widetilde{\epsi}}{2\max\pa{\normm{\avx}^2/3+\norm{\epsi}{\infty},1}}\right[$. If the event $\calE_{\rm conH}$ holds true, then for any $x,z\in B(\avx,\rho)$ or $x,z\in B(-\avx,\rho)$, we have 
	\begin{align}\label{eq:Localconvex_G}
		D_{f}(x,z)\geq\para{\lambda\min\pa{\normm{\avx}^2,1}-\widetilde{\epsi}-\vrho\max\pa{\normm{\avx}^2/3+\norm{\epsi}{\infty},1}}  D_{\psi}(x,z) ,
	\end{align}
	where $\rho=\frac{1-\lambda}{\sqrt{3}}\normm{\avx}$.
\end{lemma}
\begin{proof}
	For any  $u\in\bbR^n$, we have 
	\begin{align*}
		\nabla^2f(u)&\slon \esp{\nabla^2f(u)}-\vrho\Ppa{\normm{u}^2+\normm{\avx}^2/3+\norm{\epsi}{\infty}}\Id,\\
		&\slon  6u\transp{u}+3\normm{u}^2-2\avx\transp{\avx}-\normm{\avx}^2\Id-\widetilde{\epsi}\Id-\vrho\max\pa{\normm{\avx}^2/3+\norm{\epsi}{\infty},1}\pa{\normm{u}^2+1}.
	\end{align*}
	We then obtain, for any $v\in\mathbb{S}^{n-1}$
	\begin{align*}
		\transp{v}\Ppa{\nabla^2f(u) + \vrho\max\pa{\normm{\avx}^2/3+\norm{\epsi}{\infty},1}\nabla^2\psi(u)}v
		&\geq 3\Ppa{2\Ppa{\transp{v}u}^2+\normm{u}^2}-\Ppa{2\Ppa{\transp{v}\avx}^2+\normm{\avx}^2}-\widetilde{\epsi}\Id .
	\end{align*}
	
	Let $\rho>0$ be small enough, to be made precise later. Thus for any $u=\pm\avx+\rho v$ we get
	\begin{align*}
		\transp{v}\nabla^2f(u)v + &\vrho\max\pa{\normm{\avx}^2/3+\norm{\epsi}{\infty},1}\transp{v}\nabla^2\psi(u)v\\
		&\geq 6\Ppa{\transp{v}\avx}^2+6\rho^2\pm12\rho\transp{v}\avx+3\normm{\avx}^2\pm6\rho\transp{v}\avx+3\rho^2-2\Ppa{\transp{v}\avx}^2-\normm{\avx}^2-\widetilde{\epsi} \\
		&= 4\Ppa{\transp{v}\avx}^2+9\rho^2\pm18\rho\transp{v}\avx+2\normm{\avx}^2 -\widetilde{\epsi}.\numberthis\label{eq:hess_temp}
	\end{align*}
	For the entropy $\psi$, we also have 
	\begin{align*}
		\transp{v}\nabla^2\psi(u)v&=\normm{u}^2+1+2\Ppa{\transp{v}u}^2= 2\Ppa{\transp{v}\avx}^2 +3 \rho^2 + \pm6\rho\transp{v}\avx + \normm{\avx}^2+1 . 
	\end{align*}
	At this step, the proof becomes very similar to the noiseless phase retrieval (see \cite[Lemma B.6]{godeme2023provable}). Indeed, let us observe  that we showed  that
	for any vector $u=\pm\avx+\rho v$ with $\rho=\frac{1-\lambda}{\sqrt{3}}\normm{\avx}$ we have, 
	\begin{align*}
		4\Ppa{\transp{v}\avx}^2+9\rho^2\pm18\rho\transp{v}\avx+2\normm{\avx}^2\geq  \lambda\min\pa{\normm{\avx}^2,1}\Ppa{ 2\Ppa{\transp{v}\avx}^2 +3 \rho^2 + \pm6\rho\transp{v}\avx + \normm{\avx}^2+1}     
	\end{align*}
	By replacing this result in \eqref{eq:hess_temp}, we get 
	\begin{align*}
		&\transp{v}\nabla^2f(u)v + \vrho\max\pa{\normm{\avx}^2/3+\norm{\epsi}{\infty},1}\transp{v}\nabla^2\psi(u)v\\
		&\geq \lambda\min\pa{\normm{\avx}^2,1}\Ppa{ 2\Ppa{\transp{v}\avx}^2 +3 \rho^2 + \pm6\rho\transp{v}\avx + \normm{\avx}^2+1}-\widetilde{\epsi}\\
		&\geq \Ppa{\lambda\min\pa{\normm{\avx}^2,1}-\widetilde{\epsi}}\transp{v}\nabla^2\psi(u)v
	\end{align*}
	Finally, we have that
	\[
	\transp{v}\para{\nabla^2f(x)-\Ppa{\lambda\min\pa{\normm{\avx}^2,1}-\widetilde{\epsi}-\vrho\max\pa{\normm{\avx}^2/3+\norm{\epsi}{\infty},1}}\nabla^2\psi(x)}v \geq 0
	\]
	for all $v \in \mathbb{S}^{n-1}$ and $\rho \leq \frac{1-\lambda}{\sqrt{3}}\normm{\avx}$.
	To conclude the proof, let us observe that with the prescribed bound on $\vrho$, we have
	\begin{align}\label{eq:sigma_inf}
		\sigma=\lambda\min\pa{\normm{\avx}^2,1}-\widetilde{\epsi}-\vrho\max\pa{\normm{\avx}^2/3+\norm{\epsi}{\infty},1}>\frac{\lambda\min\pa{\normm{\avx}^2,1}-\widetilde{\epsi}}{2}>0 ,
	\end{align}
	where we used Assumption~\ref{ass:SNR} on the noise. Therefore, \eqref{eq:Localconvex_G} follows simply  by invoking Lemma~\ref{lem:bregcomp}.
\end{proof}
We have the following corollary which gives a condition on the coefficient of the signal-to-noise ratio $c_s$ ensuring that the neighborhood of strong convexity $\rho$ is greater than the noise. 
\begin{corollary}\label{cor: cond_r} For any fixed $\lambda\in\left]\frac{1}{9\sqrt{2}},1\right[$  and $\vrho\in\left]0,\frac{\lambda\min\pa{\normm{\avx}^2,1}-\widetilde{\epsi}}{2\max\pa{\normm{\avx}^2/3+\norm{\epsi}{\infty},1}}\right[$, if Assumption~\ref{ass:SNR} is satisfied then $r^2=\rho^2-\frac{4\normm{\epsi}^2}{m\sigma} > 0$, where $\rho=\frac{1-\lambda}{\sqrt{3}}\normm{\avx}$.
	
\end{corollary}
\begin{proof}
	To have the desired result, it suffices that $\rho^2-\frac{4\normm{\epsi}^2}{m\sigma}>0$ \ie $\rho>2\frac{\norm{\epsi}{\infty}}{\sqrt{\sigma}}$. From \eqref{eq:sigma_inf} we have, 
	\[
	\sqrt{\sigma}>\frac{\sqrt{\lambda\min\pa{\normm{\avx}^2,1}-\widetilde{\epsi}}}{\sqrt{2}}
	\]
	thus, 
	\begin{align*}
		\frac{2\norm{\epsi}{\infty}}{\sqrt{\sigma}}\leq\frac{2\sqrt{2}\norm{\epsi}{\infty}}{\sqrt{\lambda\min\pa{\normm{\avx}^2,1}-\widetilde{\epsi}}}\leq  \frac{2\sqrt{2}c_s\min\pa{\normm{\avx}^2,1}}{\sqrt{\lambda\min\pa{\normm{\avx}^2,1}-\widetilde{\epsi}}}.
	\end{align*}
	Therefore it suffices to show that
	\[
	\rho>\frac{2\sqrt{2}c_s\min\pa{\normm{\avx}^2,1}}{\sqrt{\lambda\min\pa{\normm{\avx}^2,1}-\widetilde{\epsi}}}.
	\]
	Replacing now $\rho$ by its expression, we get that $c_s$ should satisfy  
	\[
	(1-\lambda)\normm{\avx}\sqrt{\lambda\min\pa{\normm{\avx}^2,1}-\widetilde{\epsi}}>2\sqrt{6}c_s\min\pa{\normm{\avx}^2,1}
	\]
	which holds thanks to Assumption~\ref{ass:SNR}. 
\end{proof}
\begin{remark}
	This result estimates the maximum signal-to-noise ratio  for which we ensure that the neighborhood of strong convexity around the true vectors is well-defined. Let us notice that a more practical upper bound is 
	\begin{align}\label{eq:prat_bnd}
		c_s<\frac{(1-\lambda)\sqrt{\lambda}}{2\sqrt{6}}\leq \frac{1}{9\sqrt{2}}.
	\end{align}
	Indeed, it is a simple maximization of the function $\lambda\mapsto\frac{(1-\lambda)\sqrt{\lambda}}{2\sqrt{6}}$ over $\left]\frac{1}{9\sqrt{2}},1\right[$.  
\end{remark}

\subsection{Spectral initialization}\label{Prpro:spectralinit}

We now show that the initial guess $\xo$ generated by spectral initialization (Algorithm~\ref{alg:algoSPIniit}) belongs to a small $f$-attentive neighborhood of $\overline{\calX}$.
\begin{lemma}\label{pro:spectralinit}
	Fix $\vrho\in ]0,1[$ and assume that for $r\in[m]$, we have $\epsi[r]\leq |\transp{a_r}\avx|^2$. If the number of samples obeys $m \geq C(\vrho)n\log n$ for some sufficiently large constant $C(\vrho) > 0$ and \eqref{eq:conhess_G} holds true then  $\xo$ satisfies:\\
	
	\begin{enumerate}[label=(\roman*)]
		\item \label{pro:spectralinit-i} 
		$\dist(\xo,\overline{\calX}) \leq \eta_1(\vrho)\normm{\avx}$, 
		where 
		\begin{alignat}{2}
			\eta_1\colon ]0,1[ & \to {}&& ]0,1[ \nonumber\\
			\vrho & \mapsto {}&& \Ppa{\sqrt{2-2\sqrt{1-\vrho(1+c_s)}} + \frac{\vrho(1+c_s)}{2}} , \label{radius}
		\end{alignat}
		which is an increasing function.
		
		\item \label{pro:spectralinit-ii} Moreover, we have
		\begin{align}
			f(\xo)\leq f(\avx)+\Ppa{(1+\vrho)c_s\normm{\avx}+ L\frac{\Theta(\eta_1(\vrho)\normm{\avx})}{2}\eta_1(\vrho)}\eta_1(\vrho)\normm{\avx}^2,
		\end{align} with $
		L=3+\widetilde{\epsi}+\vrho\max\pa{\normm{\avx}^2/3+\norm{\epsi}{\infty},1}.
		$
		\item \label{pro:spectralinit-iii} Besides, for $\lambda \in ]0,1[$, if 
		\begin{equation}\label{eq:spectralinitvrhobnd_nsy}
		\vrho\leq\eta_1^{-1}\Ppa{\frac{1-\lambda}{\sqrt{3\para{6\Ppa{1+\frac{\pa{1-\lambda}^2}{3}-\frac{4\normm{\epsi}^2}{m\sigma\normm{\avx}^2}} + 1}}}\frac{1}{\max\Ppa{\normm{\avx},1}}} ,
		\end{equation} 
		then we have  $\xo \in B\Ppa{\overline{\calX},\frac{r}{\max\Ppa{\sqrt{\Theta(r)},1}}}$
		where $r^2=\frac{\pa{1-\lambda}^2}{3}\normm{\avx}^2-\frac{4\normm{\epsi}^2}{m\sigma}$. 
	\end{enumerate}
\end{lemma}

\begin{proof}
	\begin{enumerate}[label=(\roman*)]
		\item 
		Denote the matrix
		\[
		Y=\som{y[r]a_r\transp{a_r}}=\som{\pa{|\transp{a_r}\avx|^2+\epsi[r]}a_r\transp{a_r}}.
		\]
		We have \whp 
		\[
		\normm{Y-\esp{Y}}\leq\vrho\pa{\normm{\avx}^2+\norm{\epsi}{\infty}}\leq\vrho\pa{1+c_s}\normm{\avx}^2.
		\]
		Let $\wtilde{x}$ be the eigenvector associated with the largest eigenvalue $\wtilde{\lambda}$ of $Y$ such that $\normm{\wtilde{x}}=\normm{\avx}$ (obviously $\wtilde{\lambda}$ is nonnegative since $Y$ is semidefinite positive). Then, 
		\begin{align*}
			\vrho\pa{1+c_s}\normm{\avx}^4&\geq\abs{\transp{\wtilde{x}}\Ppa{Y-2\avx\transp{\avx}-\normm{\avx}^2\Id-\widetilde{\epsi}\Id}\wtilde{x}} \\
			&= \abs{\wtilde{\lambda}\normm{\avx}^2 -2\pa{\transp{\wtilde{x}}\avx}^2-\normm{\avx}^4-\widetilde{\epsi}\normm{\avx}^2}
		\end{align*}
		Hence 
		\[
		2\pa{\transp{\wtilde{x}}\avx}^2 \geq \wtilde{\lambda}\normm{\avx}^2 -\normm{\avx}^4-\widetilde{\epsi}\normm{\avx}^2-\vrho\pa{1+c_s}\normm{\avx}^4.
		\]
		Moreover, since $\wtilde{\lambda}$ is the largest eigenvalue of $Y$, applying the concentration inequality at $\avx$  \whp
		\begin{align*}
			\wtilde{\lambda}\normm{\avx}^2 \geq \transp{\avx}Y\avx &\geq \transp{\avx}\Ppa{2\avx\transp{\avx}+\normm{\avx}^2\Id-\widetilde{\epsi}\Id}\avx - \vrho\pa{1+c_s}\normm{\avx}^4\\
			&=3\normm{\avx}^4+\widetilde{\epsi}\normm{\avx}^2-\vrho\pa{1+c_s}\normm{\avx}^4.
		\end{align*}
		Merging the last two inequalities, we get
		\begin{align*}
			2\pa{\transp{\wtilde{x}}\avx}^2 &\geq 3\normm{\avx}^4+\widetilde{\epsi}\normm{\avx}^2-\vrho\pa{1+c_s}\normm{\avx}^4-\normm{\avx}^4-\widetilde{\epsi}\normm{\avx}^2-\vrho\pa{1+c_s}\normm{\avx}^4\\
			&=2\normm{\avx}^4-2\vrho\pa{1+c_s}\normm{\avx}^4.    
		\end{align*}
		Which implies that
		\[
		\dist(\wtilde{x},\overline{\calX}) \leq \sqrt{2 - 2\sqrt{1-\vrho\pa{1+c_s}}}\normm{\avx} .
		\]
		By definition of $\xo$ in Algorithm~\ref{alg:algoSPIniit}, $\xo=\sqrt{\frac{1}{m}\sum_r y[r]} \frac{\wtilde{x}}{\normm{\avx}}=\sqrt{\frac{1}{m}\sum\limits_r\pa{|\transp{a_r}\avx|^2+\epsi[r]}} \frac{\wtilde{x}}{\normm{\avx}}$, and thus \whp
		\[
		\normm{\xo-\wtilde{x}} = \abs{\sqrt{\frac{m^{-1}\sum_r y[r]}{\normm{\avx}^2}} - 1}\normm{\avx} = \abs{\sqrt{\frac{m^{-1}\sum_r\pa{|\transp{a_r}\avx|^2+\epsi[r]}}{\normm{\avx}^2}} - 1}\normm{\avx}\leq \frac{\vrho\pa{1+c_s}\normm{\avx}}{
			2},
		\]
		it comes out that,
		\[
		\dist(\xo,\overline{\calX}) \leq \dist(\wtilde{x},\overline{\calX}) + \normm{\xo-\wtilde{x}} \leq \Ppa{\sqrt{2-2\sqrt{1-\vrho(1+c_s)}} + \frac{\vrho(1+c_s)}{2}}\normm{\avx} .
		\]
		
		\item Under our sampling complexity bound, event $\calE_{\rm conH}$ defined by \eqref{eq:uniconcen_G} holds true \whp. It then follows from Lemma~\ref{pro:Lsmad_G} applied at $\avx$ and $\xo$, that
		\begin{align*}
			D_{f}(\xo,\avx) \leq L D_{\psi}(\xo,\avx).
		\end{align*}
		The latter implies that
		\begin{align*}
			&f(\xo)\leq f(\avx)+\pscal{\nabla f(\avx),\xo-\avx}+ L D_{\psi}(\xo,\avx),\\
			&f(\xo)\leq f(\avx)+\normm{\nabla f(\avx)}\normm{\xo-\avx}+ L\frac{\Theta(\eta_1(\vrho)\normm{\avx})}{2}\normm{\xo-\avx}^2,
		\end{align*}
		Since  $\nabla f(\avx)=\frac{1}{m}\sum\limits_{r=1}^{m}\epsi[r]a_r\transp{a_r}\avx$, we obtain from \eqref{eq:2ndconc} that
		\[
		\normm{\nabla f(\avx)}\leq \norm{\epsi}{\infty}\normm{\frac{1}{m}\sum\limits_{r=1}^{m}a_r\transp{a_r}}\normm{\avx}\leq(1+\vrho)\norm{\epsi}{\infty}\normm{\avx}. 
		\]
		Combining the two last inequality yields to 
		\begin{align*}
			f(\xo)\leq f(\avx)+\Ppa{(1+\vrho)c_s\normm{\avx}+ L\frac{\Theta(\eta_1(\vrho)\normm{\avx})}{2}\eta_1(\vrho)}\eta_1(\vrho)\normm{\avx}^2.
		\end{align*}
		
		\item In view of \ref{pro:spectralinit-i}, it is sufficient to show that $\eta_1(\vrho)\normm{\avx} \leq \frac{r}{\max\Ppa{\sqrt{\Theta(r)},1}}$. Since from Proposition~\ref{pro:bregman}\ref{pro:bregman4},  we have 
		\[
		\Theta(r) \leq 6(\normm{\avx}^2+r^2) + 1 \leq \para{6\Ppa{1+\frac{\pa{1-\lambda}^2}{3}-\frac{4\normm{\epsi}^2}{m\sigma\normm{\avx}^2}} + 1}\max\Ppa{\normm{\avx}^2,1} ,
		\]
		and $\eta_1$ is an increasing function, we conclude.
	\end{enumerate}
\end{proof}

\section{Landscape of the Noise-Aware Objective with Gaussian Measurements}\label{landscapestudy}  
\subsection{Warm up: Critical points of $\esp{f}$}
We start by studying and characterizing the set of critical points of $\esp{f}$. This can be seen as the asymptotic behavior of the critical points of $f$ when the number of measurements $m$ grows to $+\infty$. 
\begin{proposition}\label{pro: crit_E(f)}We have 
	\[ 
	\crit(\esp{f}) = \ens{0} \cup  \overline{\calX}_\epsi \cup\enscondlr{x\in\bbR^n}{\transp{\avx}x=0, \normm{x}^2=\frac{1}{3}\Ppa{\normm{\avx}^2+\widetilde{\epsi}}},
	\]
	where $\overline{\calX}_\epsi\eqdef\ensB{\pm\avx\sqrt{1+\frac{\widetilde{\epsi}}{3\normm{\avx}^2}}}$. Those sets are respectively, the local maximizer, the set of global minimizes, and strict saddle points of $\esp{f}$.
\end{proposition}
Before proving this result, we the closed form expressions of the expectation of $f$ and its derivatives.
\begin{lemma}\label{espCDP} For all $x\in\bbR^n$, we have: 
	\begin{align*}
		&\esp{f(x)}=\frac{3}{4}\para{\normm{x}^4+\normm{\avx}^4}-\frac{1}{2}\normm{\avx}^2\normm{x}^2-(|\transp{\avx}x|^2+\frac{\normm{\epsilon}^2}{4m}-\frac{\widetilde{\epsi}\pa{\normm{x}^2-\normm{\avx}^2}}{2},\\
		&\nabla\esp{f(x)}=3\normm{x}^2x-2\avx\pa{\transp{\avx}x}-\normm{\avx}^2x-\widetilde{\epsi}x,\\ 
		&\nabla^2\esp{f(x)}=3\Ppa{2x\transp{x}+\normm{x}^2\Id}-2\avx\transp{\avx}-\normm{\avx}^2\Id-\widetilde{\epsi}\Id. 
	\end{align*}
\end{lemma}
\begin{proof} 
	By linearity of the expectation, we have $\esp{f(x)}=\esp{f_{\mathrm{NL}}(x)}+ \esp{f_{\mathrm{Ny}}(x)}$. Linearity again yields
	\begin{align*}
		\esp{f_{\mathrm{Ny}}(x)}&=-\hsom{\epsi[r]\Ppa{\transp{x}\esp{a_r\transp{a_r}}x-\transp{\avx}\esp{a_r\transp{a_r}}\avx}}+\frac{\normm{\epsi}^2}{4m},\\
		&=-\hsom{\epsi[r]\para{\normm{x}^2-\normm{\avx}^2}}+\frac{\normm{\epsi}^2}{4m}.\numberthis\label{eq:efn}
	\end{align*}
	We also have
	\begin{align*}
		\esp{f_{\mathrm{NL}}(x)}=\qsom{\esp{|\transp{a_r}x|^4}}+\qsom{\esp{|\transp{a_r}\avx|^4}}-\hsom{|\transp{a_r}x|^2|\transp{a_r}\avx|^2}.
	\end{align*}
	From \cite[Lemma B.1]{godeme2023provable}, we know that 
	\[
	\forall x\in\bbR^n, \quad\esp{\Som{|\transp{a_r}x|^2}a_r\transp{a_r}}=2x\transp{x}+\normm{x}^2\Id. 
	\]
	Therefore we have 
	\[
	\esp{\qsom{|\transp{a_r}x|^4}}=\transp{x}\Ppa{\esp{\qsom{|\transp{a_r}x|^2}a_r\transp{a_r}}}x=\frac{1}{4}\transp{x}\Ppa{2x\transp{x}+\normm{x}^2}x=\frac{3}{4}\normm{x}^4,
	\]
	and
	\[
	\esp{\hsom{|\transp{a_r}x|^2|\transp{a_r}\avx|^2}}=\transp{x}\Ppa{\avx\transp{\avx}+\frac{1}{2}\normm{\avx}^2}x=|\transp{\avx}x|^2+\frac{1}{2}\normm{\avx}^2\normm{x}^2.
	\]
	Whence we have 
	\begin{align}\label{eq:efnl}
		\esp{f_{\mathrm{NL}}(x)}=\frac{3}{4}\para{\normm{x}^4+\normm{\avx}^4}-\frac{1}{2}\normm{\avx}^2\normm{x}^2-|\transp{\avx}x|^2.
	\end{align}
	
	The claim follows simply by summing  \eqref{eq:efn} and \eqref{eq:efnl}.
	
	We deduce the gradient and the hessian by straightforward derivation of $\esp{f}$.
\end{proof}

\begin{proof}(Proposition~\ref{pro: crit_E(f)})
	\begin{itemize}
		\item \textbf{The origin :} Let us observe that $\nabla\esp{f(0)}=0$, it follows that the value of the hessian at zero satisfies
		\begin{align*}
			\nabla^2\esp{f(0)}&=-\normm{\avx}^2\Id-2\avx\transp{\avx}-\widetilde{\epsi}\Id \lon -\pa{\normm{\avx}^2+\widetilde{\epsi}}\Id-2\avx\transp{\avx}\prec0,
		\end{align*}
		where we have used that $\widetilde{\epsi}$ is non-negative (see Assumption~\ref{ass:SNR}).  It follows that $0$ is a local maximizer of $\esp{f}$. 
		\item \textbf{ Subspaces with no critical points:} 
		\begin{enumerate}
			\item For any point in $\ensB{x\in\bbR^n: 0<\normm{x}^2<\frac{\normm{\avx}^2+\widetilde{\epsi}}{3}}$, we have 
			\begin{align*}
				\pscal{x,\nabla\esp{f(x)}}&=\para{3\normm{x}^2-\normm{\avx}^2-\widetilde{\epsi}}\normm{x}^2-2|\transp{\avx}x|^2 \leq \para{3\normm{x}^2-\normm{\avx}^2-\widetilde{\epsi}}\normm{x}^2 < 0 .
			\end{align*}
			We deduce that, $\pscal{x,\nabla\esp{f(x)}}<0$ which implies that $\normm{\nabla\esp{f(x)}}$ is bounded away from zero on this region, and thus that there are no critical points there.
			
			\item Consider a point in $\ensB{x\in\bbR^n: \frac{\normm{\avx}^2+\widetilde{\epsi}}{3} < \normm{x}^2<\normm{\avx}^2+\frac{\widetilde{\epsi}}{3}}$ and recall that it is a critical point if and only if 
			\begin{equation}\label{eq:relaD2}
			\para{3\normm{x}^2-\normm{\avx}^2-\widetilde{\epsi}}x=2(\avx\transp{\avx})x.
			\end{equation}
			Combining Assumption~\ref{ass:SNR} and the fact that $\normm{x}^2>\frac{\normm{\avx}^2+\widetilde{\epsi}}{3}$, we get that $\para{3\normm{x}^2-\normm{\avx}^2-\widetilde{\epsi}}/2$ is the positive eigenvalue of the rank-one matrix matrix $\avx\transp{\avx}$. This is equivalent to $3\normm{x}^2-\normm{\avx}^2-\widetilde{\epsi}=2\normm{\avx}^2$, \ie, $\normm{x}^2=\normm{\avx}^2+\frac{\widetilde{\epsi}}{3}$ which contradicts the definition of this region, showing again that there are no critical points there. 
			
			\item For a point in the region $\ensB{x\in\bbR^n: \normm{x}^2>\normm{\avx}^2+\frac{\widetilde{\epsi}}{3}}$, we have the lower bound
			\begin{align*}
				\pscal{x,\nabla\esp{f(x)}}&=\para{3\normm{x}^2-\normm{\avx}^2-\widetilde{\epsi}}\normm{x}^2-2|\transp{\avx}x|^2,\\
				&\geq \para{3\normm{x}^2-3\normm{\avx}^2-\widetilde{\epsi}}\normm{x}^2>0.
			\end{align*}
			Hence, $\normm{\nabla\esp{f(x)}}$ is bounded away from zero on this region yielding the same conclusion.
		\end{enumerate}

		\item \textbf{Strict saddle points:} A point in the sphere $\ensB{x\in\bbR^n: \normm{x}^2=\frac{\normm{\avx}^2+\widetilde{\epsi}}{3}}$ is a critical point if it is orthogonal to the true vector $\avx$. Indeed we have, 
		\[
		\nabla\esp{f(x)}=0 \iff \Ppa{3\normm{x}^2-\normm{\avx}^2-\widetilde{\epsi}}x = 2\avx\transp{\avx}x \iff \transp{\avx}x=0.
		\]
		Besides, for any $v\in\bbR^{n}$ we have 
		\begin{align*}
			\pscal{v,\nabla^2\esp{f(x)}v}&=6|\transp{v}x|^2+3\normm{x}^2\normm{v}^2-2|\transp{v}\avx|^2-\normm{\avx}^2\normm{v}^2-\widetilde{\epsi}\normm{v}^2 \\
			&=6|\transp{v}x|^2-2|\transp{v}\avx|^2+\normm{v}^2\Ppa{3\normm{x}^2-\normm{\avx}^2-\widetilde{\epsi}} \\
			&=6|\transp{v}x|^2-2|\transp{v}\avx|^2.
		\end{align*}
		In the direction $v=x$, we deduce that 
		\begin{align*}
			\pscal{x,\nabla^2\esp{f(x)}x}=6\normm{x}^4>0 , 
		\end{align*}
		and in the direction $v=\avx$ we have
		\begin{align*}
			\pscal{\avx,\nabla^2\esp{f(x)}\avx}=-2\normm{\avx}^4<0 ,
		\end{align*}
		where we have used orthogonality of $x$ and $\avx$. These facts show that the critical points in this region, \ie points orthogonal to $\avx$, are strict saddle points of $\esp{f}$.

		\item \textbf{Global minimizers:} In view of the above, local/global minimizers can only occur on the sphere $\ensB{x\in\bbR^n: \normm{x}^2=\normm{\avx}^2+\frac{\widetilde{\epsi}}{3}}$. Any point in his set is a critical point of $\esp{f}$ if and only if 
		\begin{align*}
			\nabla\esp{f(x)}=0 \iff & \para{3\normm{x}^2-\normm{\avx}^2-\widetilde{\epsi}}x=2\avx\transp{\avx}x \iff \para{\normm{\avx}^2\Id-\avx\transp{\avx}}x=0.  
		\end{align*}
		Therefore, $x$ is a critical point on this region if and only if $x$ is an eigenvector of $\avx\transp{\avx}$, that is $x \in \Span(\avx)$, or equivalently, $\exists \beta \in \bbR$ such that $x=\beta\avx$ with
		\begin{align*}
			\normm{x}^2=\beta^2\normm{\avx}^2=\normm{\avx}^2+\frac{\widetilde{\epsi}}{3} \iff \beta=\pm \sqrt{1+\frac{\widetilde{\epsi}}{3\normm{\avx}^2}}.
		\end{align*}
		The set of critical points in this region is reduced to $\overline{\calX}_\epsi\eqdef\ensB{\pm\avx\sqrt{1+\frac{\widetilde{\epsi}}{3\normm{\avx}^2}}}$. For $x\in\overline{\calX}_\epsi$ we have
		\begin{align*}
			\nabla^2\esp{f(x)}&=6x\transp{x}-2\avx\transp{\avx}+
			\Ppa{3\normm{x}^2-\normm{\avx}^2-\widetilde{\epsi}}\Id\\
			&=(6\beta^2-2)\avx\transp{\avx}+ 2\normm{\avx}^2\Id \succeq 2\normm{\avx}^2\Id.
		\end{align*}
		Indeed, we have
		\[
		6\beta^2-2=4+\frac{2\widetilde{\epsi}}{\normm{\avx}^2}\geq 4 > 0 ,
		\]
		where we use again the non-negativity of $\widetilde{\epsi}$ in Assumption~\ref{ass:SNR}. We conclude that $\overline{\calX}_\epsi$ is the set of global minimizers of $\esp{f}$. 
	\end{itemize}
\end{proof}

\subsection{Main result: Critical points of $f$}
In this section, we study the landscape of the objective function $f$ for the Gaussian measurement model.  Our main result hereafter characterizes the set of critical points of $f$ for $m$ large enough.
\begin{theorem}\label{thm: Landscapethm}\textbf{(Critical points of $f$)} Fix $\lambda\in\left]\frac{1}{9\sqrt{2}},1\right[$. Let us assume that the noise vector satisfies Assumption~\ref{ass:SNR}. If $m \gtrsim n\log(n)^3$, then
	\begin{equation}\label{eq:crit_charac}
	\crit(f)=\Argmin(f) \cup \strisad(f) 
	\end{equation}
	where $\Argmin(f)=\ens{\pm \xsol}$. This holds with probability of at least $1-\frac{c}{m}$ where $c$ is a positive numerical constant.
\end{theorem}

\begin{remark}\label{rem: Landscapethm} {\ }
	\begin{itemize}
		\item In \cite[Theorem~2.2]{sun_geometric_2018}, the authors study the geometry of $f$ in the noiseless case here coined as $f_{\mathrm{NL}}$. We aim with our result to extend it to the noisy case with small enough noise (see Assumption~\ref{ass:SNR}).
		
		
		\item This result shows that when the number of measurements $m$ is sufficiently large and the noise $\epsi$ is very small compared to the  true vector which is entailed  by a large SNR, then the set of critical points of the objective function $f$ is reduced to the set of global minimizers $\Argmin(f)$ and the set of strict saddle points $\strisad(f)$. The strict saddle avoidance of mirror descent  will then imply that the sequence provided by mirror descent will always converge to global minimizers of the function $f$. 
	\end{itemize} 
\end{remark}

We recall the radius $\rho=\frac{1-\lambda}{\sqrt{3}}\normm{\avx}$ defined in Lemma~\ref{pro:Localconvex_G}. To prove Theorem~\ref{thm: Landscapethm}, we consider the following regions of $\bbR^n$ which are helpful to characterize the landscape of $f$:
\begin{align}
	&\calR_1=\ensB{x\in\bbR^n: \pscal{\avx,\esp{\nabla^2f(x)}\avx}\leq -\frac{1}{100}\normm{x}^2\normm{\avx}^2-\frac{1}{50}\normm{\avx}^4},\label{eq:calR_1}\\
	&\calR_3=\ensB{x\in\bbR^n: \dist(x,\overline{\calX})\leq \rho},\label{eq:calR_2}\\
	&\calR_2=(\calR_1\cup\calR_3)^c.\label{eq:calR_3} 
\end{align}
We also define specific regions $\calR_2^x$ and $\calR_2^h$,
\begin{align*}
	\calR_2^x&=\ensB{x\in \bbR^n: \pscal{x,\esp{\nabla f(x)}}\geq \frac{1}{500}\normm{x}^2\normm{\avx}^2+\frac{1}{100}\normm{x}^4 },\\
	\calR_2^h=&\ensB{x\in\bbR^n: \pscal{d_x,\esp{\nabla f(x)}}\geq\frac{1}{250}\normm{x}\normm{d_x}\normm{\avx}^2,\frac{11}{20}\normm{\avx}\leq\normm{x}\leq\normm{\avx},\dist(x,\overline{\calX})\geq\frac{\normm{\avx}}{3}}. 
\end{align*}
where $d_x=\frac{\pm\avx-x}{\normm{\pm\avx-x}}$ if $x \neq \pm\avx$ and any vector on the unit sphere otherwise.

Let us observe that these regions are similar to those defined in \cite{sun_geometric_2018} replacing $f$ by $f_{\mathrm{NL}}$.
Indeed, the idea  behind  our assumptions  is the fact  that small noise will introduce small perturbations  in the function $f$, and therefore under our assumption of small noise, the latter has benign influence on the landscape of $f$ (see Figure~\ref{fig: fonction landscape}). Mainly, in the region $\calR_1$ the function $f$  still has negative curvature. In the region $\calR_2$, $f$ has a large gradient and in $\calR_3$ relative strong convexity with respect to our chosen entropy $\psi$. It is important to observe that in the noisy case, it is not true that the true vectors $\overline{\calX}$ are critical points of $f$ or even $\esp{f}$. However, we have already shown in Lemma~\ref{lem:interm} that $\pm \avx$ are actually $\frac{\normm{\epsi}^2}{m}$-minimizers. Moreover, we have already given in Proposition~\ref{pro: crit_E(f)} a  description of the set of critical points of $\esp{f}$, providing a hint that in the large oversampling regime, the geometry of $f$ is close to that of $f_{\mathrm{NL}}$. This result shows that the set of critical points of $\esp{f}$ is also reduced to the set of strict saddle points with symmetric global minimizers of $\esp{f}$. This set of minimizers, that we denoted $\overline{\calX}_\epsi$ (see Proposition~\ref{pro: crit_E(f)}), are direct perturbations of the true vectors $\overline{\calX}$ by the noise; see also Lemma~\ref{lem:xbar_et_x} which quantifies the distance of global minimizers of $f$ to $\overline{\calX}$ in probability. 
\begin{figure}[htbp]
	\centering
	\includegraphics[trim={0cm 8cm 0 8cm},clip,width=0.8\linewidth]{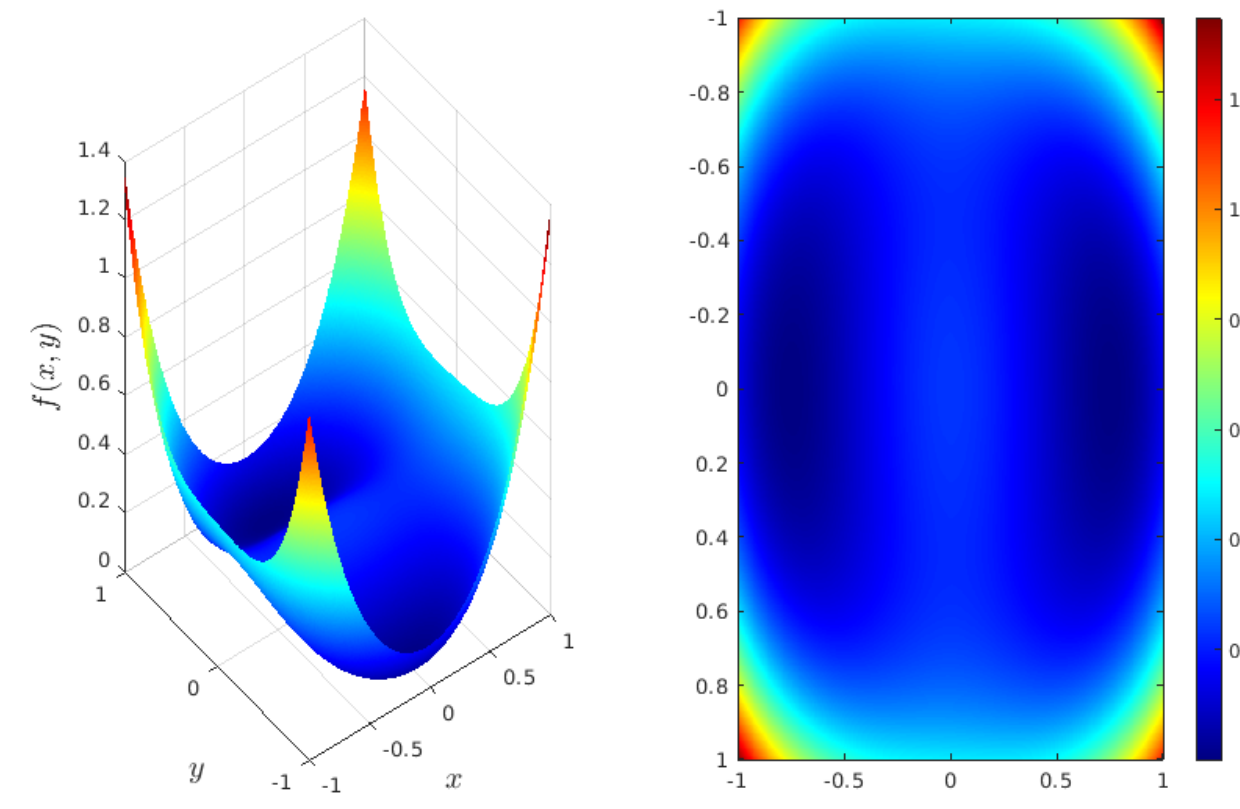}
	\caption{Landscape of the function $f$ as $m\to\infty$; we have $(m,n)=(200,2)$ and the true vectors are $[\pm3/4,0]$. The noise vector is generated at uniform in [-1,1] such that $\widetilde{\epsi}\approx 5. 10^{-3}$.  One  clearly sees that the geometry of the landscape of $f$ is preserved and that the only minimizers of $f$ are  very close to the true vectors. }
	\label{fig: fonction landscape}
\end{figure}

\begin{proof} 
	In the following, all assertions are to be understood in high probability sense. The proof consists in invoking properly the statements of Proposition~\ref{pro:Noisyland}. 
	In the region $\calR_1$, Proposition~\ref{pro:Noisyland}-\ref{pro:Noisyland:item1} shows that 
	\begin{equation}\label{eq:str_R1}
	\forall x\in \calR_1, \quad \pscal{\avx,\nabla^2f(x)\avx}\leq-\frac{1}{100}\pa{1-c_s}\normm{\avx}^4,
	\end{equation}
	\ie $f$ has a negative curvature in the direction of the true vectors $\xoverline{\calX}$ which means that any critical point in $\calR_1$ is a strict saddle point for $f$. From Proposition~\ref{pro:Noisyland}-\ref{pro:Noisyland:item3} and \ref{pro:Noisyland:item4}, we deduce that 
	\begin{equation}\label{eq:prooftheo}
	\forall x\in \calR_2^x\cup \calR_2^{h}, \quad \normm{\nabla f(x)} \geq \frac{1}{500}\pa{1-c_s }\normm{x}\normm{\avx}^2. 
	\end{equation}
	Moreover, Proposition~\ref{pro:Noisyland}-\ref{pro:Noisyland:item5} entails that  $\calR_2\subset\calR_2^x\cup\calR_2^h$ which means that \eqref{eq:prooftheo} holds true for all $x\in\calR_2$. Thus the gradient of function $f$ is bounded away from zero on $\calR_2$ which means that there are no critical points in this region. 
	Therefore, local/global minimizers of $f$ are necessarily located in the region $\calR_3$. It remains to show that the only critical points in the domain $\calR_3$ are just the elements of $\Argmin(f)$\footnote{Remember that $\Argmin(f)$ is a nonempty compact set by injectivity of $A$ under the assumed measurement bound.} which contains only two points $\pm \xsol$. This will be a consequence of $\sigma$-strong convexity of $f$ on $\calR_3$. In the following, since $\calR_3 = B(\avx,\rho) \cup B(-\avx,\rho)$, we prove the claim only on $B(\avx,\rho)$ and the same holds for the symmetric case with $-\avx$. Let $x \in B(\avx,\rho)\backslash\{\xsol\}$. In view of Proposition~\ref{pro:Noisyland}-\ref{pro:Noisyland:item2}, we have
	\begin{align*}
		D_{f}(x,\xsol)&= f(x) - \min f \geq \sigma D_{\psi}(x,\xsol) \geq \frac{\sigma}{2} \normm{x-\xsol}^2 .
	\end{align*} 
	The right hand side is positive since $x \neq \xsol$, which means that $f$ has a unique minimizer on $B(\avx,\rho)$. Moreover,
	\begin{align*}
		D_{f}(\xsol,x)&= \min f - f(x) - \pscal{\nabla f(x),\xsol-x} \geq \sigma D_{\psi}(\xsol,x) \geq \frac{\sigma}{2} \normm{x-\xsol}^2 ,
	\end{align*} 
	and thus
	\[
	\pscal{\nabla f(x),x-\xsol} \geq D_{f}(\xsol,x) \geq \frac{\sigma}{2} \normm{x-\xsol}^2 .
	\]
	Cauchy-Schwarz then entails
	\[
	\normm{\nabla f(x)} \geq \frac{\sigma}{2} \normm{x-\xsol} > 0 
	\]
	meaning that $f$ has no other critical point than $\xsol$ on $B(\avx,\rho)$. This completes the proof.
\end{proof}

The proof of the above result heavily relies on the behaviour of $f$ on each region. This is the subject of the next proposition.


\begin{proposition}\label{pro:Noisyland} If the number of samples obeys $m\gtrsim n\log^3(n)$ then with probability $1-\frac{c}{m}$ where $c$ is a positive numerical constant, we have the following statements. 
	\begin{enumerate}[label=(\roman*)]
		\item \label{pro:Noisyland:item1} 
		In the region $\calR_1,$ the objective $f$  has a negative curvature \ie,
		\begin{equation}\label{eq:negcurv}
		\forall x\in\calR_1,\quad\pscal{\avx,\nabla^2f(x)\avx}\leq-\frac{1}{100}\pa{1-c_s}\normm{\avx}^4.
		\end{equation}
		\item \label{pro:Noisyland:item2} 
		In $\calR_3$, $f$ is $\sigma$-strongly convex where $\sigma > 0$ is given in Proposition~\ref{pro:Localconvex_G}.
		
		\item \label{pro:Noisyland:item3} 
		The gradient is bounded from away from zero in  $\mathcal{R}_2^x$. More precisely, 
		\begin{equation}
		\forall x\in\calR_2^x,\quad  \pscal{x,\nabla f(x)} \geq \frac{1}{1000}\pa{1-c_s }\normm{x}^2\normm{\avx}^2.
		\end{equation}
		\item \label{pro:Noisyland:item4}  
		We have
		\begin{equation}
		\forall x\in\calR^h_2,\quad  \pscal{d_x,\nabla f(x)}\geq\frac{1}{500}\pa{1-c_s }\normm{\avx}^2\normm{x}\normm{d_x}.
		\end{equation}
		\item \label{pro:Noisyland:item5} We have $
		\calR_2\subset\calR_2^x\cup\calR_2^h.$
	\end{enumerate}
\end{proposition}
\begin{remark} The previous result extends the series of propositions (\cite[Proposition~2.3-2.7]{sun_geometric_2018}) to the noisy case. All the statements depend on the (inverse) signal-to-noise coefficient $c_s$ which obviously less than $1$ under our assumption.  In the noiseless case, let us  observe that we recover all the Propositions mentioned above. 
\end{remark}

\begin{proof}$~$
	\begin{enumerate}[label=(\roman*)] 
		
		\item[\ref{pro:Noisyland:item1}] For any $x\in\mathcal{R}_1$, we have 
		\begin{equation*}
			\pscal{\avx,\nabla^2f(x)\avx}=\som{3\abs{\transp{a_r}x}^2\abs{\transp{a_r}\avx}^2}-\som{\abs{\transp{a_r}\avx}^4}-\som{\epsi[r]\abs{\transp{a_r}\avx}^2}.
		\end{equation*}
		By using similar concentration inequalities as in Lemma~\ref{lem:conhess_G} we have the following 
		\begin{align*}
			&\som{3\abs{\transp{a_r}x}^2\abs{\transp{a_r}\avx}^2}\leq \esp{\som{3\abs{\transp{a_r}x}^2\abs{\transp{a_r}\avx}^2}}+\vrho\normm{\avx}^2\normm{x}^2,\\
			&\som{\abs{\transp{a_r}\avx}^4}\geq \esp{\som{\abs{\transp{a_r}\avx}^4}}-\vrho\normm{\avx}^4,\\
			& \som{\epsi[r]\abs{\transp{a_r}x}^2}\geq \esp{\som{\epsi[r]\abs{\transp{a_r}x}^2}}-\vrho\norm{\epsi}{\infty}\normm{x}^2.
		\end{align*}
		After summing, we get
		\begin{equation*}
			\pscal{\avx,\nabla^2f(x)\avx}\leq \pscal{\avx,\esp{\nabla^2f(x)}\avx}+\vrho\normm{x}^2\normm{\avx}^2+\vrho\normm{\avx}^4+\vrho c_s\normm{\avx}^2\normm{x}^2.
		\end{equation*}
		We choose now  $\vrho=\frac{1}{100}$,  and  since $x\in\calR_1$, we finally obtain that 
		\begin{align*}
			\pscal{\avx,\nabla^2f(x)\avx}&\leq-\frac{1}{100}\normm{x}^2\normm{\avx}^2-\frac{1}{50}\normm{\avx}^4+ \frac{1}{100}\normm{x}^2\normm{\avx}^2+\frac{1}{100}\normm{\avx}^4+\frac{1}{100} c_s\normm{\avx}^2\normm{x}^2,\\
			&=-\frac{1}{100}\pa{1-c_s}\normm{\avx}^4. 
		\end{align*}
		
		\item[\ref{pro:Noisyland:item2}] Combine Lemma~\ref{pro:Localconvex_G} and $1-$strong convexity of $\psi$. 
		
		\item[\ref{pro:Noisyland:item3}]Let $x\in\calR_2^x,$
		\begin{align*}
			\pscal{x,\nabla f(x)}=\som{\abs{\transp{a_r}x}^4}-\som{\abs{\transp{a_r}\avx}^2\abs{\transp{a_r}x}^2}-\som{\epsi[r]\abs{\transp{a_r}x}^2}.
		\end{align*}
		using the same concentration arguments as in the proof of Lemma~\ref{lem:conhess_G}, we get 
		\begin{align*}
			\pscal{x,\nabla f(x)}\geq&\pscal{x,\esp{\nabla f(x)}}-\frac{1}{100}\normm{x}^4-\frac{1}{1000}\normm{x}^2\normm{\avx}^2-\frac{\normm{x}^2\norm{\epsilon}{\infty}}{1000},\\
			\geq& \frac{1}{500}\normm{x}^2\normm{\avx}^2+\frac{1}{100}\normm{x}^4 -\frac{1}{100}\normm{x}^4-\frac{1}{1000}\normm{x}^2\normm{\avx}^2-\frac{\normm{x}^2\norm{\epsilon}{\infty}}{1000},\\
			=&\frac{1}{1000}(1-c_s)\normm{\avx}^2\normm{x}^2 ,
		\end{align*}  
		where we used Assumption~\ref{ass:SNR} in the last inequality.
		\item[\ref{pro:Noisyland:item4}] We have, 
		\begin{align*}
			\pscal{d_x,\nabla f(x)}=\pscal{d_x,\nabla f_{\mathrm{NL}}(x)}+\pscal{d_x,\nabla f_{\mathrm{Ny}}(x)}.
		\end{align*}
	On the one hand, we have
	\begin{align*}
		\pscal{d_x,\nabla f_{\mathrm{NL}}(x)}=\som{\abs{\transp{a_r}x}^2\pa{\transp{d_x}a_r\transp{a_r}x}}-\som{\abs{\transp{a_r}\avx}^2\pa{\transp{d_x}a_r\transp{a_r}x}}.
	\end{align*}	
		Therefore, the proof bears exactly the same path as the proof of \cite[Proposition 2.6]{sun_geometric_2018}. Indeed using Lemma~\ref{lem:conhess_G} more precisely \cite[Lemma~B.2]{godeme2023provable}, we have for $m \gtrsim n\log(n)$ 
		\begin{align*}
			\som{\abs{\transp{a_r}\avx}^2\pa{\transp{d_x}a_r\transp{a_r}x}}&\leq \transp{d_x}\esp{\som{\abs{\transp{a_r}\avx}^2a_r\transp{a_r}}}x+\frac{1}{1000}\normm{\avx}^2\normm{d_x}\normm{x}, 
		\end{align*}
	with probability larger than $1-\frac{C}{m}$, $C > 0$.

To bound the first term, we use the line of reasoning used in the proof of \cite[Proposition~2.6]{sun_geometric_2018} to get that when $m \gtrsim n\log(n)^3$, with a probability at least $1-\frac{C'}{m}$, $C' > 0$, we have 	
		\begin{align*}
		\som{\abs{\transp{a_r}x}^2\pa{\transp{d_x}a_r\transp{a_r}x}}&\geq \transp{d_x}\esp{\som{\abs{\transp{a_r}x}^2a_r\transp{a_r}}}x-\frac{1}{1000}\normm{\avx}^2\normm{d_x}\normm{x}, 
	\end{align*}
we refer to \cite[Section~6.5]{sun_geometric_2018} for the details of the proofs. Then, provided that $m \gtrsim n\log(n)^3$, with a union bound, we have with probability at least $1-\frac{C+C'}{m}$ that	
		\begin{align*}
			\pscal{d_x,\nabla f_{\mathrm{NL}}(x)}\geq \pscal{d_x,\esp{\nabla f_{\mathrm{NL}}(x)}}-\frac{1}{500}\normm{\avx}^2\normm{x}\normm{d_x} .
		\end{align*}
		On the other hand, with similar arguments as in the proof of Lemma~\ref{lem:conhess_G} and using again Assumption~\ref{ass:SNR}, we have with probability larger than $1-\frac{C+C'}{m}$
		\begin{align*}
			\pscal{d_x,\nabla f_{\mathrm{Ny}}(x)}&=\transp{d_x}\Ppa{\som{\epsi[r]a_r\transp{a_r}}}x
			\geq{\pscal{d_x,\esp{\nabla f_{\mathrm{Ny}}(x)}}}-\frac{c_s}{500}\normm{\avx}^2\normm{x}\normm{d_x}.
		\end{align*}
		We combine now the last two inequalities to get 
		\begin{align*}
			\pscal{d_x,\nabla f(x)}\geq \pscal{d_x,\esp{\nabla f(x)}}-\frac{1}{500}\normm{\avx}^2\normm{x}\normm{d_x}-\frac{c_s}{500}\normm{\avx}^2\normm{x}\normm{d_x}
		\end{align*}
		with probability at least $1-2\frac{C+C'}{m}$. Using the definition of the region $\calR_2^h$ we obtain the claimed result.
				
		\item[\ref{pro:Noisyland:item5}] The proof is similar to that of \cite[Proposition 2.7]{sun_geometric_2018} which consists of showing that  $\bbR^n=\calR_1\cup\calR_2^x\cup\calR_2^h\cup\calR_3$. We then get out claim by definition of $\calR_2$ and that $\bbR^n=\calR_1\cup\calR_2\cup\calR_3$. The idea is to divide the $\bbR^n$ into several overlapping regions and show that we can cover them with our good partition. To achieve this task we will use the set
		\begin{align*}
			\calR_2^{h'}=\ensB{x\in\bbR^n: \pscal{d_x,\esp{\nabla f(x)}}\geq \frac{1}{250}\normm{\avx}^2\normm{x}\normm{d_x}, \normm{x}\leq\normm{\avx}}.
		\end{align*}
		
		\begin{itemize}
			\item We can cover the set $\calR_a\eqdef\ensB{x\in\bbR^n:|\transp{x}\avx|\leq\frac{1}{2}\normm{x}\normm{\avx}}$ with both $\calR_1$ and $\calR_2^x$. If $\normm{x}^2\leq\frac{298}{451}\normm{\avx}^2$
			\begin{align*}
				\pscal{\avx,\nabla^2f(x)\avx}+\frac{1}{100}\normm{\avx}^2\normm{x}^2&=6\pa{\transp{x}\avx}^2+\frac{301}{300}\normm{x}^2\normm{\avx}^2-3\normm{\avx}^4-\widetilde{\epsi}\normm{\avx}^2\\
				&\leq\Ppa{\frac{3}{2}+\frac{301}{100}}\normm{\avx}^2\normm{x}^2-\frac{149}{50}\normm{\avx}^4-\widetilde{\epsi}\normm{\avx}^2-\frac{1}{50}\normm{\avx}^4\\
				&\leq \frac{298}{100}\normm{\avx}^4-\frac{149}{50}\normm{\avx}^4-\widetilde{\epsi}\normm{\avx}^2-\frac{1}{50}\normm{\avx}^4\\
				&\leq \frac{1}{50}\normm{\avx}^4.
			\end{align*}
			If $\normm{x}^2\geq\frac{626}{995}\normm{\avx}^2$, 
			\begin{align*}
				\pscal{x,\esp{\nabla f(x)}}-\frac{1}{500}\normm{x}^2\normm{\avx}^2&=3\normm{x}^4-2\pa{\transp{x}\avx}^2-\frac{501}{500}\normm{\avx}^2\normm{x}^2-\widetilde{\epsi}\normm{x}^2\\
				&\geq 3\normm{x}^4-\frac{1}{2}\normm{x}^2\normm{\avx}^2-\frac{501}{500}\normm{\avx}^2\normm{x}^2-\widetilde{\epsi}\normm{x}^2\\
				&\geq \frac{1}{100}\normm{x}^4+\frac{299}{100}\normm{x}^4-\frac{751}{500}\normm{\avx}^2\normm{x}^2-\widetilde{\epsi}\normm{x}^2\\
				&\geq \frac{1}{100}\normm{x}^4+\frac{299}{100}\normm{x}^4-\Ppa{\frac{751}{500}+\frac{1}{9\sqrt{2}}}\Ppa{\frac{995}{626}}\normm{x}^4\\
				&\geq \frac{1}{100}\normm{\avx}^4,
			\end{align*}
			where we have used the fact that $\widetilde{\epsi}\geq 0$  combined with the practical upper bound on $c_s$ \eqref{eq:prat_bnd}. Since $\frac{298}{451}\geq \frac{626}{995}$, we  conclude that $\calR_a\subset\calR_1\cup\calR_2^x$.
			\item The set $\calR_b\eqdef\ensB{x\in\bbR^n: |\transp{x}\avx|\geq\frac{1}{2}\normm{x}\normm{\avx}; \normm{x}\leq\frac{57}{100}\normm{\avx}}$ is covered by the set $\calR_1$. Indeed for any $x\in\calR_b$ we have, 
			\begin{align*}
				\pscal{\avx,\nabla^2f(x)\avx}+\frac{1}{100}\normm{\avx}^2\normm{x}^2&=6\pa{\transp{x}\avx}^2+\frac{301}{300}\normm{x}^2\normm{\avx}^2-3\normm{\avx}^4-\widetilde{\epsi}\normm{\avx}^2\\
				&\leq\frac{901}{100}\normm{x}^2\normm{\avx}^2-\frac{149}{50}\normm{\avx}^4-\widetilde{\epsi}\normm{\avx}^2-\frac{1}{50}\normm{\avx}^4\\
				&\leq\frac{901}{100}\Ppa{\frac{57}{100}}^2\normm{\avx}^2-\frac{149}{50}\normm{\avx}^4-\widetilde{\epsi}\normm{\avx}^2-\frac{1}{50}\normm{\avx}^4\\
				&\leq -\widetilde{\epsi}\normm{\avx}^2-\frac{1}{50}\normm{\avx}^4\leq -\frac{1}{50}\normm{\avx}^4,
			\end{align*}
			since  $\widetilde{\epsi}\geq0$.
			\item Let consider the set $\calR_c\eqdef\ensB{ x\in\bbR^n:\frac{1}{2}\normm{x}\normm{\avx}\leq|\transp{x}\avx|\leq \frac{99}{100}\normm{x}\normm{\avx}}$, which is covered by $\calR^x_2$ and $\calR_2^{h'}$. For any $x\in \calR_c$ such that $\normm{x}\geq \sqrt{\frac{1996}{1973}}\normm{\avx}$ we have 
			\begin{align*}
				\pscal{x,\esp{\nabla f(x)} }-&\frac{1}{500}\normm{\avx}^2\normm{x}^2+\frac{1}{100}\normm{x}^4\\
				&=3\normm{x}^4-2\pa{\transp{x}\avx}^2-\normm{\avx}^2\normm{x}^2-\widetilde{\epsi}\normm{x}^2-\frac{1}{500}\normm{\avx}^2\normm{x}^2\\
				&\geq \frac{299}{100}\normm{x}^4+\frac{501}{500}\normm{\avx}^2\normm{x}^2-2\pa{\transp{x}\avx}^2-2\pa{\transp{x}\avx}^2-\widetilde{\epsi}\normm{x}^2\\
				&\geq \frac{299}{100}\normm{x}^4-\Ppa{2\Ppa{\frac{99}{100}}^2+\frac{501}{500}+\frac{1}{9\sqrt{2}}}\normm{\avx}^2\normm{x}^2\\
				&\geq \Ppa{\frac{299}{100}\Ppa{\frac{1996}{1973}}^2-\Ppa{2\Ppa{\frac{99}{100}}^2+\frac{501}{500}+\frac{1}{9\sqrt{2}}}}\normm{x}^4\\
				&\geq 0.
			\end{align*}
			Therefore, we have $\calR_c\cap\ensB{x\in\bbR^n:\normm{x}\geq \sqrt{\frac{1996}{1973}}\normm{\avx}}\subset\calR^x_2$. To show the remaining inclusion,  we use an  $(\alpha,\beta)-$type argument. 
			Let assume that $\normm{x}=\alpha\normm{\avx}$, $|\transp{x}\avx|=\beta=\normm{\avx}\normm{x}=\alpha\beta\normm{\avx}$ and $\widetilde{\epsi}=\eps\normm{\avx}^2$ with $\alpha\in\left[\frac{11}{20},\sqrt{\frac{1996}{1973}}\right]$, $\beta\in\left[\frac{1}{2},\frac{99}{100}\right]$ and $\eps\in\left[0,\frac{1}{9\sqrt{2}}\right]$. We have
			\begin{align*}
				\pscal{x-\avx,\esp{\nabla f(x)}}&=3\normm{x}^4+3\pa{\transp{x}\avx}\Ppa{\normm{\avx}^2-\normm{x}^2}-2\pa{\transp{x}\avx}^2-\normm{\avx}^2\normm{x}^2+\widetilde{\epsi}\Ppa{\pa{\transp{x}\avx}-\normm{x}}\\
				&=\normm{\avx}^4\alpha\Ppa{3\alpha^3+3\beta\pa{1-\alpha^2}-2\alpha\beta^2-\alpha+\eps\pa{\beta-\alpha}} .
			\end{align*}
			Whence we have 
			\begin{multline*}
				\frac{1}{\normm{\avx}^4\alpha}\Ppa{\pscal{x-\avx,\esp{\nabla f(x)}}-\frac{1}{250}\normm{\avx}^2\normm{x}\normm{d_x}}=3\alpha^3+3\beta\pa{1-\alpha^2}-2\alpha\beta^2\\
				-\alpha+\eps\pa{\beta-\alpha}-\frac{1}{250}\sqrt{1+\alpha^2-2\alpha\beta} .
			\end{multline*}
			It is straightforward that in this domain $\frac{1}{250}\sqrt{1+\alpha^2-2\alpha\beta}\leq \frac{1}{250}\sqrt{\frac{3969-2\sqrt{984527}}{1973}}\leq \frac{41}{10000}$. Therefore we  define the following  function
			\[
			p(\alpha,\beta,\eps)\eqdef3\alpha^3+3\beta\pa{1-\alpha^2}-2\alpha\beta^2 -\alpha+\eps\pa{\beta-\alpha}-\frac{41}{10000}
			\]
			Then $p$ has a unique minimizer arising at $\left(0.998237,\frac{99}{100},\frac{1}{9\sqrt{2}}\right)$  with a value  $\frac{87239}{2500000}$. We deduce that
			\begin{align*}
				\pscal{x-\avx,\esp{\nabla f(x)}}-&\frac{1}{250}\normm{\avx}^2\normm{x}\normm{d_x}\geq0
			\end{align*}
			Therefore $\calR_c\subset\calR_2^x\cup\calR_2^{h'}$.
			\item We now cover $\calR_d\eqdef\ensB{x\in\bbR^n:\frac{99}{100}\normm{\avx}\normm{x}\leq |\transp{x}\avx|\leq\normm{x}\normm{\avx}, \normm{x}\geq\frac{11}{20}\normm{\avx}}$ with $\calR_2^x,\calR_3$ and $\calR_2^{h'}$. 
			For any $x\in\calR_d,$ with $\normm{x}\geq\sqrt{\frac{1031}{1000}}\normm{\avx}$, we have 
			\begin{align*}
				\pscal{x,\esp{\nabla f(x)}}-\frac{1}{500}\normm{x}^2\normm{\avx}^2-&\frac{1}{100}\normm{\avx}^4=\frac{299}{100}\normm{x}^4-2\pa{\transp{x}\avx}^2-\frac{501}{500}\normm{\avx}^2\normm{x}^2-\widetilde{\epsi}\normm{x}^2\\
				&\geq \frac{299}{100}\normm{x}^4-\frac{1501}{500}\normm{\avx}^2\normm{x}^2-\widetilde{\epsi}\normm{x}^2\\
				&\geq \frac{299}{100}\normm{x}^4-\Ppa{\frac{1501}{500}+\frac{1}{9\sqrt{2}}}\normm{\avx}^2\normm{x}^2\\
				&\geq \Ppa{\frac{299}{100}-\Ppa{\frac{1501}{500}+\frac{1}{9\sqrt{2}}}\frac{1000}{1031}}\normm{x}^4 \geq 0.
			\end{align*}
			hence we have $\calR_d\cap\ensB{x\in\bbR^n:\normm{x}\geq\sqrt{\frac{1031}{1000}}\normm{\avx}}\subset \calR_2^x$.
			When $\frac{23}{25}\leq\normm{x}\leq\sqrt{\frac{1031}{1000}}$ we have, 
			\begin{align*}
				\normm{d_x}^2&=\normm{\avx}^2+\normm{x}^2-2\transp{x}\avx=\normm{\avx}^2\Ppa{1+\alpha^2-2\alpha\beta}, 
			\end{align*}
			where we have $\alpha\in\left[\frac{23}{25},\sqrt{\frac{1031}{1000}}\right]$ and $\beta\in\left[\frac{99}{100},1\right]$.  Therefore, we consider
			\begin{align*}
				p(\alpha,\beta,\lambda)=1+\alpha^2-2\alpha\beta-\frac{(1-\lambda)^2}{3},
			\end{align*}
			with $\lambda\in\left]\frac{1}{9\sqrt{2}},\frac{3}{5}\right]$. The maximum value of $p$ is taken at $(\alpha,\beta,\lambda)=\left(\frac{23}{25},\frac{99}{100},\frac{3}{5}\right)$ thus
			$p(\alpha,\beta,\lambda)\leq-\frac{107}{3750}$. We deduce that $\calR_d\cap\ensB{x\in\bbR^n,\frac{23}{25}\leq\normm{x}\leq\sqrt{\frac{1031}{1000}}}\subset \calR_3$. When $\frac{11}{20}\normm{\avx}\normm{x}\leq\frac{24}{25}\normm{\avx}$, we have 
			\begin{multline*}
				\frac{1}{\normm{\avx}^4\alpha}\Ppa{\pscal{x-\avx,\esp{\nabla f(x)}}-\frac{1}{250}\normm{\avx}^2\normm{x}\normm{d_x}}=3\alpha^3+3\beta\pa{1-\alpha^2}-2\alpha\beta^2\\
				-\alpha+\eps\pa{\beta-\alpha}-\frac{1}{250}\sqrt{1+\alpha^2-2\alpha\beta} 
			\end{multline*}
			where $\alpha\in\left[\frac{11}{20},\sqrt{\frac{24}{25}}\right]$, $\beta\in\left[\frac{99}{100},1\right]$ and $\eps\in\left[0,\frac{1}{9\sqrt{2}}\right].$ One check easily that 
			\[
			p(\alpha,\beta,\eps)=3\alpha^3+3\beta\pa{1-\alpha^2}-2\alpha\beta^2
			-\alpha+\eps\pa{\beta-\alpha}-\frac{1}{250}\sqrt{1+\alpha^2-2\alpha\beta}\geq 0.
			\]
			Consequently, we have $\calR_d\cap\ensB{x\in\bbR^n: \frac{11}{20}\normm{\avx}\normm{x}\leq\frac{24}{25}\normm{\avx}}\subset\calR_2^{h'}$. Finally $\calR_d\subset\calR_2^x\cup\calR_3\cup\calR_2^{h'}$. 
		\end{itemize}
		By construction, we have $\calR_a\cup\calR_b\cup\calR_c\cup\calR_d=\bbR^n$, and therefore 
		\begin{align*}
			\bbR^n&=\calR_a\cup\calR_b\cup\calR_c\cup\calR_d\\
			&\subset \calR_1\cup\calR_2^x\cup\calR_2^{h'}\cup\calR_3\\
			&= \calR_1\cup\calR_2^x\cup\Ppa{\calR_2^{h'}\cap\ensB{x\in\bbR^n:\frac{11}{20}\normm{\avx}\leq \normm{x}}}\cup\calR_3\\
			&= \calR_1\cup\calR_2^x\cup\Ppa{\calR_2^{h'}\cap\ensB{x\in\bbR^n:\frac{11}{20}\normm{\avx}\leq \normm{x}}\cap\calR_3^c}\cup\calR_3\\
			&= \calR_1\cup\calR_2^x\cup\calR_2^h\cup\calR_3.
		\end{align*}
	\end{enumerate} 
\end{proof}                                                      
    
\end{appendices}

\begin{acknowledgments}
The authors thank the French National Research Agency (ANR) for funding the project FIRST (ANR-19-CE42-0009). 
\end{acknowledgments}

\bibliographystyle{plain}
{
\bibliography{biblio/biblio3}
}

\end{document}